\documentclass{birkau}
\usepackage[T1]{fontenc}
\usepackage{mathrsfs}
\usepackage{enumitem}
\usepackage{amsmath}
\usepackage{amsthm}
\usepackage{amssymb}

\makeatletter
\theoremstyle{plain}
\newtheorem{theorem}{\protect\theoremname}[section]
\theoremstyle{definition}
\newtheorem{definition}[theorem]{\protect\definitionname}
\theoremstyle{plain}
\newtheorem{proposition}[theorem]{\protect\propositionname}
\theoremstyle{definition}
\newtheorem{remark}[theorem]{\protect\remarkname}
\theoremstyle{plain}
\newtheorem{corollary}[theorem]{\protect\corollaryname}
\theoremstyle{plain}
\newtheorem{lemma}[theorem]{\protect\lemmaname}
\theoremstyle{definition}
\newtheorem{example}[theorem]{\protect\examplename}
\theoremstyle{plain}
\newtheorem{question}[theorem]{\protect\questionname}

\numberwithin{equation}{section}
\usepackage{amsmath,amsfonts}
\DeclareMathOperator{\Trim}{Trim}
\DeclareMathOperator{\id}{id}
\DeclareMathOperator{\PI}{PI}
\DeclareMathOperator{\Inv}{Inv}
\DeclareMathOperator{\Co}{Co}
\setlist[enumerate,1]{label = (\alph*)}

\makeatother

\providecommand{\corollaryname}{Corollary}
\providecommand{\definitionname}{Definition}
\providecommand{\examplename}{Example}
\providecommand{\lemmaname}{Lemma}
\providecommand{\propositionname}{Proposition}
\providecommand{\questionname}{Question}
\providecommand{\remarkname}{Remark}
\providecommand{\theoremname}{Theorem}

\begin{document}
\title{Partitions of primitive Boolean spaces}
\author{Andrew B. Apps}
\date{February 2025}

\address{Independent researcher\\St Albans, UK}
\email{andrew.apps@apps27.co.uk}
\thanks{I gratefully acknowledge the use of Cambridge University's library facilities.}
\begin{abstract}
A Boolean ring and its Stone space (Boolean space) are \emph{primitive}
if the ring is disjointly generated by its \emph{pseudo-indecomposable}
(PI) elements. Hanf showed that a primitive PI Boolean algebra can
be uniquely defined by a \emph{structure diagram}. In a previous paper
we defined \emph{trim $P$-partitions} of a Stone space, where $P$
is a \emph{PO system} (poset with a distinguished subset), and showed
how they provide a physical representation within the Stone space
of these structure diagrams\emph{.}

In this paper we study the class of trim partitions of a fixed primitive
Boolean space, which may not be compact, and show how they can be
structured as a quasi-ordered set via an appropriate refinement relation.
This refinement relation corresponds to a surjective morphism of the
associated PO systems, and we establish a quasi-order isomorphism
between the class of well-behaved partitions of a primitive space
and a class of \emph{extended PO systems}. 

We also define \emph{rank partitions}, which generalise the \emph{rank
diagrams} introduced by Myers, and the \emph{ideal completion} of
a trim $P$-partition, whose underlying PO system is the ideal completion
of $P$, and show that rank partitions are just the ideal completions
of trim partitions. In the process, we extend a number of existing
results regarding primitive Boolean algebras or compact primitive
Boolean spaces to locally compact Boolean spaces. 
\end{abstract}
\maketitle
\section{Introduction}

Primitive Boolean algebras were introduced by Hanf~\cite{Hanf} as
a class of well-behaved Boolean algebras, and they and their Stone
spaces have been studied using a range of topological and algebraic
methods. In a previous paper~\cite{Apps-Stone} we defined \emph{trim
$P$-partitions} of a locally compact primitive Boolean space (i.e.
the Stone space of a Boolean ring, which may not have a $1$), where
$P$ is a PO system (poset with a distinguished subset), and showed
how they correspond to and provide a physical representation within
the Stone space of Hanf's ``structuring functions\emph{'', }with
``trim'' sets providing a topological equivalent of the algebraic
concept of a pseudo-indecomposable set\emph{.} Thus a second countable
locally compact Stone space (which we term an \emph{$\omega$-Stone
space}) admits a trim partition iff it is primitive. 

This paper continues the study of trim partitions of primitive Boolean
spaces, and the interplay with the order properties of associated
PO systems. We establish a quasi-order isomorphism between the class
of well-behaved trim partitions of a primitive space, ordered by a
``regular refinement'' relation, and a class of extended PO systems,
ordered by surjective morphisms. We also define \emph{rank partitions},
which generalise the \emph{rank diagrams} defined by Myers, and the
\emph{ideal completion} of a trim $P$-partition, whose underlying
PO system is the ideal completion of $P$, and show that rank partitions
are just the ideal completions of trim partitions. In the process,
we extend several existing results regarding primitive Boolean algebras
or compact primitive Boolean spaces to locally compact Boolean spaces.

An element $A$ of a Boolean ring $R$ (which may or may not have
a $1$ and whose Stone space therefore may or may not be compact)
is \emph{pseudo-indecomposable (PI)} if for all $B\in(A)$, either
$(B)\cong(A)$ or $(A-B)\cong(A)$, where $(A)=\{C\in R\mid C\subseteq A\}$;
and $R$ and its Stone space are \emph{primitive }if every element
of $R$ is the disjoint union of finitely many PI elements. 

If $W$ is a primitive $\omega$-Stone space, we can associate with
it a canonical PO system $\mathscr{S}(W)$, whose elements are the
homeomorphism classes of the PI compact open subsets of $W$, with
a suitable relation. Hanf showed that a countable primitive PI Boolean
algebra is determined up to isomorphism by this associated canonical
PO system, which ``structures'' the algebra.

We can further define a canonical trim $\mathscr{S}(W)$-partition
$\mathscr{X}(W)$ of $W$. For $p\in\mathscr{S}(W)$, the associated
partition element $X_{p}$ contains all points of $W$ that have a
neighbourhood base of compact opens homeomorphic to $p$, and the
order relation on $\mathscr{S}(W)$ satisfies 

\begin{equation}
X_{p}^{\prime}\cap X=\bigcup_{q<p}X_{q},\label{eq:P-partition}
\end{equation}

where $X_{p}^{\prime}$ denotes the limit points of $X_{p}$ and $X=\bigcup_{q\in\mathscr{S}(W)}X_{q}$.
We note that $X$ contains precisely the \emph{homogeneous points}
defined by Pierce~\cite[Section~3.2]{PierceMonk}, and is a dense
subset of $W$.

The results of Williams~\cite{Touraille primitive} and others, translated
into the language of Stone space partitions, mean that if $W$ is
a compact primitive PI $\omega$-Stone space then $\mathscr{S}(W)$
is a simple ``diagram'' (PO system with a minimal element such that
any morphism from it to another PO system is injective); and if $Q$
is any other countable diagram whose simple image is isomorphic to
$\mathscr{S}(W)$, then $W$ admits a trim $Q$-partition. So compact
primitive PI $\omega$-Stone spaces can be classified by countable
simple PO systems with a minimal element.

In order to extend this classification result to spaces that may be
neither compact nor PI, we need to consider \emph{extended PO systems}
$(P,L,f)$ and trim $(P,L,f)$-partitions, where $L\subseteq P$ defines
which partition elements are relatively compact, and the function
$f$ specifies the size of the finite partition elements. We showed
in~\cite{Apps-Stone} that if we restrict to \emph{bounded} $\omega$-Stone
spaces~$W$, where the union of the relatively compact elements of
$\mathscr{X}(W)$ is relatively compact, then $W$ is determined up
to homeomorphism by the extended PO system associated with $\mathscr{X}(W)$.
Hence we can classify bounded $\omega$-Stone spaces in terms of extended
PO systems $(P,L,f)$ for which $P$ is simple.

We develop this further in Section~\ref{sec:Trim partitions} by
considering the class of trim partitions of a general fixed primitive
Boolean space. We are interested in when one trim partition ``regularly''
refines or consolidates another in a way that respects the trim sets,
and it turns out that trim partitions that regularly consolidate a
trim $Q$-partition correspond to surjective morphisms $\alpha\colon Q\rightarrow P$
for some PO system $P$. For bounded spaces $W$, this yields an isomorphism
between the set of homeomorphism classes of ``bounded'' trim partitions
of $W$ and the set of isomorphism classes of extended PO systems
that arise for such partitions, both structured as quasi-ordered systems.
We also identify a necessary and sufficient condition on a PO system
$Q$ for there to be a trim $Q$-partition that regularly refines
a given bounded trim partition of a space.

In Section~\ref{sec:Complete-partitions} we turn to complete ``semi-trim''
$P$-partitions of Stone spaces, which extend a trim partition to
cover the entire space by adding its ``limit points'', so that~(\ref{eq:P-partition})
is still satisfied with $X=W$. Myers~\cite{Myers} defined a \emph{rank
diagram }of a Stone space and showed that a compact $\omega$-Stone
space has a rank diagram iff it is primitive and PI, and that every
such space is ranked by its \emph{orbit diagram} (the set of orbits
in $W$ under homeomorphisms of $W$, with an associated relation).
We extend Myers' definition and results to \emph{rank partitions}
of general primitive spaces, and make precise the link between the
rank diagrams of Myers and trim and semi-trim partitions (and hence
with the structure diagrams of Hanf). Specifically, we identify conditions
on a PO system $P$ for a complete semi-trim $P$-partition to be
a rank partition; we define the \emph{ideal completion} of a trim
$P$-partition, whose underlying PO system is the ideal completion
of $P$; and show that rank partitions correspond precisely to ideal
completions of trim partitions. Hansoul~\cite{Hansoul} had previously
shown this for the special case where $W$ is compact and PI and $P$
is the canonical PO system for~$W$: in this instance the ideal completion
of the canonical trim partition of $W$ is just the orbit diagram
of $W$.

Finally, we extend to locally compact spaces results firstly of Hansoul
regarding the smallest complete topological Boolean sub-algebra (\emph{prime
TBA}) of $2^{W}$ for primitive~$W$, and secondly of Touraille characterising
primitive spaces by means of an algebraic property of the associated
Heyting{*}-algebra.

We follow most subsequent authors after Hanf in dropping the requirement
that a primitive Boolean algebra be PI\@. We refer to~\cite{Hansoul3}
for a brief overview of existing work on primitive Boolean algebras
and to~\cite[Section~3]{PierceMonk} for a detailed exposition.

\section{\label{sec:definitions}Trim partitions: definitions and basic properties}
\begin{definition}[\emph{PO systems}]
We recall that a \emph{PO system }is a set $P$ with an anti-symmetric
transitive relation $<$; equivalently, it is a poset with a distinguished
subset $P_{1}=\{p\in P\mid p<p\}$. We write $p\leqslant q$ to mean
$p<q$ or $p=q$; and $P^{d}$ for $\{p\in P\mid p\nless p\}$. If
$P$ is a poset, we write $p\lneqq q$ if $p\leqslant q$ and $p\neq q$.

Let $P$ be a poset or PO system and let $Q\subseteq P$. We write
$Q_{\downarrow}=\{p\in P\mid p\leqslant q\text{ for some }q\in Q\}$,
$Q_{\uparrow}=\{p\in P\mid p\geqslant q\text{ for some }q\in Q\}$,
and write $P_{\min}$ and $P_{\max}$ for the sets of minimal and
maximal elements of $(P,\leqslant)$ respectively. We write $\widetilde{P}$
for the PO system with the same elements as $P$ but with the reversed
order relation. For $p\in P$, we write $p_{\downarrow}$ and $p_{\uparrow}$
for $\{p\}_{\downarrow}$ and $\{p\}_{\uparrow}$ respectively.

We recall that $Q$ is an \emph{upper} (respectively \emph{lower})
subset of $P$ if for all $q\in Q$, whenever $r\geqslant q$ (respectively
$r\leqslant q$) then $r\in Q$.

We will say that $F\subseteq P$ is a \emph{finite foundation} of
$Q$ if $F$ is a finite subset of $Q_{\downarrow}$ such that for
all $r\in Q_{\downarrow}$ we can find $p\in F$ such that $p\leqslant r$
(equivalently, $Q_{\downarrow}\subseteq F_{\uparrow}$).

We recall that $J\subseteq P$ is an \emph{ideal} of $P$ if $J$
is a lower subset such that for all $x,y\in J$, we can find $z\in J$
such that $x\leqslant z$ and $y\leqslant z$; and that $J$ is a
\emph{principal ideal of $P$} if $J=p_{\downarrow}$ for some $p\in P$.
If $P$ is countable, the non-principal ideals of $P$ have form $J=\{p\in P\mid p\leqslant q_{n}\text{ for some }n\}$,
where $\{q_{n}\mid n\geqslant1\}$ is a strictly increasing sequence
in $P$. 
\end{definition}

\emph{Notation}: If $\{X_{p}\mid p\in P\}$ is a partition of a subset
$X$ of $W$, so that each $X_{p}\neq\emptyset$, we will write $\{X_{p}\mid p\in P\}^{*}$
for the partition $\{\{X_{p}\mid p\in P\},W-X\}$ of $W$, with the
convention that $W-X$ may be the empty set. The asterisk will be
omitted where a complete partition of $W$ is intended (i.e.\ where
$X=W$). 

For $Q\subseteq P$, we write $X_{Q}$ for $\bigcup_{p\in Q}X_{p}$.

If $W$ is a topological space and $S\subseteq W$, we write $S^{\prime}$
for the \emph{derived set }of $S$, namely the set of all limit points
of $S$, and $\overline{S}$ for the closure of $S$, so that $\overline{S}=S\cup S^{\prime}$.

We recall the following definitions from~\cite{Apps-Stone}.
\begin{definition}[\emph{Partitions}]
Let $W$ be a Stone space, $(P,<)$ a PO system, and $\mathscr{X}=\{X_{p}\mid p\in P\}^{*}$
a partition of $W$; let $X=X_{P}$. 

We will say that $\mathscr{X}$ is a \emph{$P$-partition }of $W$
if $X_{p}^{\prime}\cap X=\bigcup_{q<p}X_{q}$ for all $p\in P$.

We define the \emph{label function} $\tau_{\mathscr{X}}\colon X\rightarrow P$
by $\tau_{\mathscr{X}}(w)=p$ for $w\in X_{p}$, and write $\tau$
for $\tau_{\mathscr{X}}$ when there is no ambiguity.

For a subset $Y$ of $W$, we define its \emph{type }$T_{\mathscr{X}}(Y)=\{p\in P\mid Y\cap X_{p}\neq\emptyset\}$,
and write $T$ for $T_{\mathscr{X}}$ when there is no ambiguity. 

A compact open subset $A$ of $W$ is \emph{$p$-trim,} and we write
$t(A)=p$, if $T(A)=p_{\uparrow}$, with additionally $|A\cap X_{p}|=1$
if $p\in P^{d}$. A set is \emph{trim }if it is $p$-trim for some
$p\in P$. We write $\Trim(\mathscr{X})$ for the set of trim subsets
of $W$, $\Trim_{p}(\mathscr{X})$ for $\{A\in\Trim(\mathscr{X})\mid t(A)=p\}$
and $\widehat{P}$ for $\{t(A)\mid A\in\Trim(\mathscr{X})\}$. 

For $w\in W$, let $V_{w}=\{A\subseteq W\mid w\in A\wedge A\in\Trim(\mathscr{X})\}$
and let $I_{w}=\{t(A)\mid A\in V_{w}\}$, being the trim neighbourhoods
of $w$ and their types.

We say that $w\in W$ is a \emph{clean point }if $w\in X_{p}$ for
some $p$ and has a $p$-trim neighbourhood, and is otherwise a \emph{limit
point}.

A $P$-partition of $W$ is \emph{complete} if $X=W$.

A $P$-partition $\mathscr{X}$ of $W$ is a \emph{trim }$P$-partition
if it also satisfies:
\begin{description}
\item [{T1\label{T1}}] Every element of $W$ has a neighbourhood base
of trim sets (this implies that $X$ is dense in $W$);
\item [{T2\label{T2}}] The partition is \emph{full}: that is, for each
$p\in\widehat{P}$, every element of $W$ with a neighbourhood base
of $p$-trim sets is an element of $X_{p}$;
\item [{T3\label{T3}}] All points in $X$ are clean. 
\end{description}
If $\{Y_{p}\mid p\in P\}^{*}$ is a $P$-partition of the Stone space
$Z$, we say that $\alpha\colon W\rightarrow Z$ is a \emph{$P$-homeomorphism}
if it is a homeomorphism such that $X_{p}\alpha=Y_{p}$ for all $p\in P$. 
\end{definition}

If $\mathscr{X}$ is a $P$-partition, it is easy to see that $T(A)$
is an upper subset of $P$ for $A\in R$. Moreover, if $A\in\Trim_{p}(\mathscr{X})$,
$w\in X_{p}$, $B\in R$ and $w\in B\subseteq A$, then $B\in\Trim_{p}(\mathscr{X})$,
and so $w$ has a neighbourhood base of $p$-trim sets. We will also
use the following facts (\cite[Propositions~2.13,~4.2]{Apps-Stone})
repeatedly:
\begin{proposition}[Trim partitions]
\label{basic properties-2 trim}Let $P$ be a PO system and $\mathscr{X}=\{X_{p}\mid p\in P\}^{*}$
a trim $P$-partition of the Stone space $W$ of the countable Boolean
ring $R$.
\begin{description}
\item [{TP1\label{TP1}}] If $w\in W$, then $I_{w}$ is an ideal of $P$;
\item [{TP2\label{TP2}}] $w\in W$ is a clean point iff $I_{w}$ is a
principal ideal of $P$;
\item [{TP3\label{TP3}}] If $J$ is an ideal of $P$ and $B\in\Trim(\mathscr{X})$
with $t(B)\in J$, then we can find $w\in B$ such that $I_{w}=J$;
\item [{TP4\label{TP4}}] $P$ is countable;
\item [{TP5\label{TP5}}] If $A\in R$, there is a finite partition of
$A$ containing precisely one $p$-trim set for each $p\in F-P^{d}$,
and containing $|A\cap X_{p}|$ $p$-trim sets for $p\in F\cap P^{d}$,
where $F=T(A)_{\min}$;
\item [{TP6\label{TP6}}] If $A,B\in R$ are both $p$-trim, then $(A)\cong(B)$
and $A$ and $B$ are PI\@.
\end{description}
\end{proposition}

\section{\label{sec:Trim partitions}Primitive spaces and trim partitions}

If $R$ is a Boolean ring and $A\in R$, we write $(A)$ for $\{C\in R\mid C\subseteq A\}$,
the ideal of $R$ generated by $A$; $[A]$ for the isomorphism class
of the ideal $(A)$; $0$ for the zero element; and $1_{R}$ for the
multiplicative identity if $R$ is a Boolean algebra. If $W$ is a
Boolean space, we write $\Co(W)$ for the Boolean ring of compact
open subsets of $W$.
\begin{definition}[\emph{Boolean rings}]
If $R$ is a Boolean ring with Stone space $W$, we say that $A\in R$
is \emph{pseudo-indecomposable (PI)} if for all $B\in R$ such that
$B\subseteq A$, either $(B)\cong(A)$ or $(A-B)\cong(A)$; and that
$R$ and $W$ are \emph{primitive }if every element of $R$ is the
disjoint union of finitely many PI elements. We will say that $W$
is pseudo-indecomposable if it is compact and $1_{R}$ is PI\@.

An \emph{$\omega$-Stone space} is a Stone space arising from a countable
Boolean ring.
\end{definition}

\begin{remark}
We do not require the Boolean ring to have a $1$, and if it does
have a $1$ we do not require this to be PI\@.
\end{remark}

The importance of trim partitions stems from the following result;
we note that the ``if'' statement is a consequence of TP5 and TP6,
while ``only if'' is immediate from Theorem~\ref{Thm:canonical trim}
below.
\begin{theorem}[{Apps~\cite[Theorem~3.12]{Apps-Stone}}]
\label{Primitive=00003Dtrim}An $\omega$-Stone space is primitive
iff it admits a trim $P$-partition for some PO system $P$.
\end{theorem}

In this section we consider the class of trim partitions of a fixed
$\omega$-Stone space and the inter-relationships between them. As
we shall see, the inter-relationships are underpinned by morphisms
between the corresponding PO systems, which we now define.

\subsection{Morphisms and simple images of PO systems}

The following definitions arise in the work of Hanf~\cite{Hanf}
and Pierce~\cite{PierceMonk}. For reasons explained below, however,
we have adopted the opposite order convention to~\cite{PierceMonk}.
\begin{definition}[\emph{Morphisms and congruence relations}]
A \emph{morphism }$\alpha\colon Q\rightarrow P$ between PO systems
$Q$ and $P$ is a map $\alpha$ such that for each $q\in Q$, $\{r\in Q\mid r>q\}\alpha=\{p\in P\mid p>q\alpha\}$.

A \emph{congruence relation} $\sim$ on a PO system $Q$ is an equivalence
relation such that for all $q,r,s\in Q$, if $q<r$ and $q\sim s$
then there is $t\in Q$ such that $r\sim t$ and $s<t$. We write
$[q]$ for the equivalence class containing $q$.

The PO system $Q$ is \emph{simple} if it admits no non-trivial congruence
relations; equivalently, if any morphism $\alpha\colon Q\rightarrow P$
between PO systems is injective.
\end{definition}

\begin{remark}
\label{rem:morphisms}There is a natural duality between surjective
morphisms and congruence relations for a PO system. For if $\alpha\colon Q\rightarrow P$
is a morphism, then the relation $q\sim s$ iff $q\alpha=s\alpha$
is a congruence relation; conversely, if $\sim$ is a congruence relation,
then $\beta\colon Q\rightarrow Q\,/\sim:q\mapsto[q]$ is a morphism,
where $[q]<[r]$ iff there are $s,t\in Q$ such that $s\sim q$, $t\sim r$
and $s<t$.
\end{remark}

Every PO system has a unique simple image under some morphism:
\begin{theorem}[{Pierce~\cite{Pierce}, Hansoul~\cite[Proposition~1.8]{Hansoul2}}]
\label{Thm:simple image}Let $P$ be a PO system. 
\begin{enumerate}
\item There is a simple PO system $s(P)$ and a unique surjective morphism
$\alpha\colon P\rightarrow s(P)$;
\item $s(P)$ is unique up to isomorphism and has no non-trivial automorphisms;
\item for any surjective morphism $\beta\colon P\rightarrow Q$, there is
a unique surjective morphism $\gamma\colon Q\rightarrow s(P)$ such
that $\alpha=\beta\circ\gamma$. 
\end{enumerate}
\end{theorem}

We will term $s(P)$ the \emph{simple image} of $P$; it is just $P\,/\sim$,
where $\sim$ is the maximal congruence on $P$ (being the union of
all congruences on $P$). It can readily be shown that a general PO
system $P$ is simple iff (a) the ``diagram'' $p_{\uparrow}$ is
simple for each $p\in P$ and (b) if $p_{\uparrow}\cong q_{\uparrow}$
($p,q\in P$), then $p=q$. 

\subsection{Regular refinements and consolidations of trim partitions}
\begin{definition}
Let $P$ and $Q$ be countable PO systems, $\mathscr{Z}=\{Z_{p}\mid p\in P\}^{*}$
a $P$-partition and $\mathscr{Y}=\{Y_{q}\mid q\in Q\}^{*}$ a $Q$-partition
of the $\omega$-Stone space $W$, and let $\alpha\colon Q\rightarrow P$
be a map. We say that $\mathscr{Y}$ \emph{refines }$\mathscr{Z}$
\emph{via }$\alpha$ if $Y_{q}\subseteq Z_{q\alpha}$ for each $q\in Q$.
If also $\mathscr{Z}$ and $\mathscr{Y}$ both satisfy~T1, we say
that $\mathscr{Y}$ \emph{regularly refines }$\mathscr{Z}$ \emph{via
}$\alpha$, or that $\mathscr{Z}$ \emph{regularly consolidates }$\mathscr{Y}$\emph{
via }$\alpha$, writing $\mathscr{Y}\prec_{\alpha}\mathscr{Z}$, if
for all $p\in\widehat{P}$ and $A\in\Trim_{p}(\mathscr{Z})$ we have
$p\in T_{\mathscr{Y}}(A)\alpha$.
\end{definition}

\begin{remark}
\label{rem:Refinemt}This generalises the idea of a \emph{regular
extension }in~\cite[Definition~4.4]{Apps-Stone}, which applies when
$\alpha$ is injective: see also Section~\ref{subsec:Regular-extensions}
below. The definition can be generalised to the case where $\mathscr{Z}$
and/or $\mathscr{Y}$ do not satisfy~T1 by requiring instead that
$T_{\mathscr{Z}}(A)\subseteq(T_{\mathscr{Y}}(A)\alpha)_{\uparrow}$
for all compact open $A$.

We note that if $\mathscr{Y}$ refines $\mathscr{Z}$ via $\alpha$,
then $\alpha$ preserves the order operation. For if $q,r\in Q$ and
$q<r$, then $Y_{q}\subseteq Y_{r}^{\prime}\subseteq Z_{r\alpha}^{\prime}$;
so $Z_{q\alpha}\subseteq Z_{r\alpha}^{\prime}$ and $q\alpha<r\alpha$. 
\end{remark}

Regularity of a refinement of a trim partition is highly desirable
in order that the trim sets in the two partitions are aligned, as
illustrated by the following Proposition.
\begin{proposition}
\label{Prop:regref}Let $P$ and $Q$ be countable PO systems, $\mathscr{Z}$
a $P$-partition and $\mathscr{Y}$ a $Q$-partition of the $\omega$-Stone
space $W$ both satisfying~T1, and $\alpha\colon Q\rightarrow P$
a map such that $\mathscr{Y}$ refines $\mathscr{Z}$ via $\alpha$.
Then the following are equivalent:
\begin{enumerate}
\item \label{enu:regref1}$\mathscr{Y}$ \emph{regularly refines }$\mathscr{Z}$
\emph{via }$\alpha$;
\item \label{enu:regref2}if $A\in\Trim_{r}(\mathscr{Z})$, then we can
find $s\in\widehat{Q}$ and $B\in\Trim_{s}(\mathscr{Y})$ such that
$B\subseteq A$ and $s\alpha=r$;
\item \label{enu:regref3}$\Trim_{q}(\mathscr{Y})\subseteq\Trim_{q\alpha}(\mathscr{Z})$
for all $q\in\widehat{Q}$.
\end{enumerate}
\end{proposition}

\begin{proof}
Let $\mathscr{Z}=\{Z_{p}\mid p\in P\}^{*}$ and $\mathscr{Y}=\{Y_{q}\mid q\in Q\}^{*}$.

\ref{enu:regref1}~$\Rightarrow$~\ref{enu:regref2}: If $A\in\Trim_{r}(\mathscr{Z})$,
then we can find $t\in T_{\mathscr{Y}}(A)$ such that $t\alpha=r$.
Choose $x\in A\cap Y_{t}$ and use~T1 to find $B\in\Trim_{s}(\mathscr{Y})$,
say, where $s\in\widehat{Q}$, such that $x\in B\subseteq A$, so
that $A\cap Z_{s\alpha}\neq\emptyset$. Then $s\leqslant t$, so $r\leqslant s\alpha\leqslant t\alpha=r$,
and $r=s\alpha$ as required.

\ref{enu:regref2}~$\Rightarrow$~\ref{enu:regref3}: Suppose $A\in\Trim_{q}(\mathscr{Y})$.
Then $Y_{q}\subseteq Z_{q\alpha}$, so $A\cap Z_{q\alpha}\neq\emptyset$
and $T_{\mathscr{Z}}(A)\supseteq(q\alpha)_{\uparrow}$ as $T_{\mathscr{Z}}(A)$
is an upper subset of $P$. Choose $w\in Y_{q}\cap A$ and suppose
$(A-\{w\})\cap Z_{p}\neq\emptyset$ for some $p\in P$: say $x\in(A-\{w\})\cap Z_{p}$. 

Choose $B\in\Trim_{r}(\mathscr{Z})$, say, such that $x\in B\subseteq A-\{w\}$;
now find $s\in\widehat{Q}$ and $C\in\Trim_{s}(\mathscr{Y})$ such
that $C\subseteq B$ and $s\alpha=r$. Then $s>q$, so by Remark~\ref{rem:Refinemt}
we have $p\geqslant r=s\alpha>q\alpha$. Hence $A$ is $q\alpha$-trim
in $\mathscr{Z}$.

\ref{enu:regref3}~$\Rightarrow$~\ref{enu:regref1}: Suppose $p\in\widehat{P}$
and $A\in\Trim_{p}(\mathscr{Z})$. Clearly $T_{\mathscr{Z}}(A)\supseteq T_{\mathscr{Y}}(A)\alpha$.
Choose $x\in A\cap Z_{p}$ and $B\in\Trim_{q}(\mathscr{Y})$, say,
such that $x\in B\subseteq A$. Then $B$ is both $p$-trim and $q\alpha$-trim
in $\mathscr{Z}$, and so $p=q\alpha\in T_{\mathscr{Y}}(A)\alpha$.
\end{proof}
\begin{remark}
In~\cite[Example~4.6]{Apps-Stone} we give an example of an extension
of a trim partition that is not regular, where (in the notation above)
there are sets in $\Trim(\mathscr{Y})$ that are $p$-trim in $\mathscr{Z}$
for $p\notin Q\alpha$; in particular, $\alpha$ is not surjective
in this instance. A simple modification of this example would yield
an extension where $\alpha$ was surjective on $W$ but not on some
compact open subset $A$ of $W$. So it is not sufficient that $\alpha$
be surjective on~$W$: we need $\alpha$ to be surjective ``on each
compact open'', so that that the consolidation does not introduce
any unexpected trim partition elements anywhere.
\end{remark}

We will need the following general properties of regular refinements.
\begin{proposition}[Regular refinements]
\label{Prop:regrefprops}

Let $P$ and $Q$ be countable PO systems, $\alpha\colon Q\rightarrow P$
a map, $\mathscr{Y}$ a trim $Q$-partition of the $\omega$-Stone
space $W$, and $\mathscr{Z}$, $\mathscr{Z}_{1}$ and $\mathscr{Z}_{2}$
trim $P$-partitions of $W$.
\begin{enumerate}
\item \label{enu:refine1-1}If $\mathscr{Y}\prec_{\alpha}\mathscr{Z}$,
then $\alpha$ is a surjective morphism; 
\item \label{enu:refine1-3}if $\mathscr{Y}\prec_{\alpha}\mathscr{Z}_{1}$
and $\mathscr{Y}\prec_{\alpha}\mathscr{Z}_{2}$, then $\mathscr{Z}_{1}=\mathscr{Z}_{2}$;
\item \label{enu:refine1-4}if $\mathscr{Y}\prec_{\alpha}\mathscr{Z}$ and
$\mathscr{Z}\prec_{\gamma}\mathscr{V}$, where $N$ is a PO system,
$\gamma\colon P\rightarrow N$ and $\mathscr{V}$ is a trim $N$-partition
of $W$, then $\mathscr{Y}\prec_{\alpha\gamma}\mathscr{V}$;
\item \label{enu:refine1-5}if $\mathscr{Y}\prec_{\alpha}\mathscr{Z}$ and
$\alpha$ is an isomorphism, then $\mathscr{Y}=\mathscr{Z}$ (after
relabelling);
\item \label{enu:refine1-6}if $\mathscr{Y}\prec_{\alpha}\mathscr{Z}$ and
$\mathscr{Z}\prec_{\gamma}\mathscr{Y}$, where $\gamma\colon P\rightarrow Q$,
then $\mathscr{Y}=\mathscr{Z}$.
\end{enumerate}
\end{proposition}

\begin{proof}
\ref{enu:refine1-1} Let $\mathscr{Z}=\{Z_{p}\mid p\in P\}^{*}$ and
$\mathscr{Y}=\{Y_{q}\mid q\in Q\}^{*}$, so that $Y_{q}\subseteq Z_{q\alpha}$
for $q\in Q$. 

It follows immediately from the definition of a regular refinement
that $\alpha$ is surjective. We must show that $\alpha$ is a morphism.
If $q<r$, then $q\alpha<r\alpha$ by Remark~\ref{rem:Refinemt}.
Conversely, if $p\in P$, $q\in Q$ and $q\alpha<p$, choose $y\in Y_{q}$
and find $A\in\Trim_{q}(\mathscr{Y})$ such that $y\in A$. Then $A\in\Trim_{q\alpha}(\mathscr{Z})$,
so we can find $B\subseteq A-\{y\}$ such that $B\in\Trim_{p}(\mathscr{Z})$,
as $p>q\alpha$. By regularity, we can find $s\in Q$ such that $s\alpha=p$
and $B\cap Y_{s}\neq\emptyset$, so that $s>q$. Hence $\alpha$ is
a surjective morphism, as required. 

\ref{enu:refine1-3} Let $\mathscr{Z}_{i}=\{Z_{ip}\mid p\in P\}^{*}$
($i=1,2$). Suppose $w\in Z_{1p}$ for $p\in P$ and let $\{A_{n}\mid n\geqslant1\}\subseteq\Trim_{p}(\mathscr{Z}_{1})$
be a neighbourhood base of $w$. 

Find $q_{n}\in Q$ and $B_{n}\in\Trim_{q_{n}}(\mathscr{Y})$ such
that $w\in B_{n}\subseteq A_{n}$; by regularity, $B_{n}$ is $q_{n}\alpha$-trim
in both $\mathscr{Z}_{1}$ and $\mathscr{Z}_{2}$. But $B_{n}$ is
$p$-trim in $\mathscr{Z}_{1}$. Hence $B_{n}\in\Trim_{p}(\mathscr{Z}_{2})$,
so by fullness $w\in Z_{2p}$.

\ref{enu:refine1-4} If $q\in Q$, then $\Trim_{q}(\mathscr{Y})\subseteq\Trim_{q\alpha}(\mathscr{Z})\subseteq\Trim_{q\alpha\gamma}(\mathscr{V})$,
using Proposition~\ref{Prop:regref}.

\ref{enu:refine1-5} Let $\mathscr{Y}\alpha$ denote the $P$-partition
$\{Y_{q\alpha}\mid q\in Q\}^{*}$, which is a trim partition, just
being $\mathscr{Y}$ relabelled. Then $\mathscr{Y}\prec_{\alpha}\mathscr{Y}\alpha$
and $\mathscr{Y}\prec_{\alpha}\mathscr{Z}$, so $\mathscr{Z}=\mathscr{Y}\alpha$
by~\ref{enu:refine1-3}. 

\ref{enu:refine1-6} We see immediately that $Y_{q}\subseteq Z_{q\alpha}\subseteq Y_{q\alpha\gamma}$.
Hence $\alpha\gamma$ and $\gamma\alpha$ are both the identity, and
$Y_{q}=Z_{q\alpha}$ for all $q\in Q$.
\end{proof}
The following result, which generalises Williams~\cite[Lemma~2]{Williams},
underpins what follows.
\begin{proposition}
\label{Prop:trim image under morphism}Let $P$ and $Q$ be PO systems,
$\alpha\colon Q\rightarrow P$ a surjective morphism, and $\mathscr{Y}$
a trim $Q$-partition of the (primitive) $\omega$-Stone space $W$.
Then there is a unique trim $P$-partition of $W$, which we denote
$\mathscr{Y}\alpha$, such that $\mathscr{Y}\prec_{\alpha}\mathscr{Y}\alpha$. 
\end{proposition}

\begin{proof}
For $p\in P$ let $Z_{p}=\bigcup_{q\in Q}\{Y_{q}\mid q\alpha=p\}$
and let $\mathscr{Z}=\{Z_{p}\mid p\in P\}^{*}$. We show first that
$\mathscr{Z}$ is a $P$-partition of $W$. Choose $p\in P$ and suppose
$y\in\bigcup_{r<p}Z_{r}$; say $y\in Y_{q}\subseteq Z_{r}$, where
$q\alpha=r<p$. As $\alpha$ is a morphism, we can find $s>q$ such
that $s\alpha=p$. Hence $y\in Y_{q}\subseteq Y_{s}^{\prime}\subseteq Z_{p}^{\prime}$.
Conversely, suppose $y\in Z_{p}^{\prime}\cap Z_{P}$; say $y\in Y_{q}$.
Choose a $q$-trim neighbourhood $A$ of $y$. Then $(A-\{y\})\cap Z_{p}\neq\emptyset$;
choose $z\in(A-\{y\})\cap Z_{p}$ and suppose $z\in Y_{s}$. Then
$q<s$, and $y\in Y_{q}\subseteq Z_{q\alpha}$ with $q\alpha<s\alpha=p$.
Hence $Z_{p}^{\prime}\cap Z_{P}=\bigcup_{r<p}Z_{r}$ and $\mathscr{Z}$
is a $P$-partition of $W$.

We claim next that $\Trim_{q}(\mathscr{Y})\subseteq\Trim_{q\alpha}(\mathscr{Z})$.
For if $A\in\Trim_{q}(\mathscr{Y})$ and $p>q\alpha$, then we can
find $r>q$ such that $p=r\alpha$, so $A\cap Z_{p}\supseteq A\cap Y_{r}\neq\emptyset$;
conversely, if $(A-\{y\})\cap Z_{p}\neq\emptyset$, where $y\in A\cap Y_{q}$,
then $p=r\alpha$ for some $r\in Q$ such that $r>q$, so $p>q\alpha$,
as required. 

Hence $\Trim(\mathscr{Z})$ forms a neighbourhood base for each element
of $W$ and $\mathscr{Z}$ satisfies T1. Moreover, if $z\in Z_{p}$
for $p\in P$ then we can find $q\in Q$ such that $z\in Y_{q}$ and
$p=q\alpha$; but $\mathscr{Y}$ is a trim $Q$-partition, so $z$
has a $q$-trim neighbourhood, which will also be $p$-trim; hence
all points in $Z_{p}$ are clean and $\mathscr{Z}$ satisfies T3. 

Now for $p\in P$, let $U_{p}$ consist of all elements of $W$ having
a neighbourhood base of $p$-trim sets, and let $\mathscr{Y}\alpha=\{U_{p}\mid p\in P\}^{*}$,
so that $Z_{p}\subseteq U_{p}$. It is easy to see that $\mathscr{\mathscr{Y}\alpha}$
satisfies T1 to T3 and has the same trim sets as $\mathscr{Z}$.
Finally, $\Trim_{q}(\mathscr{Y})\subseteq\Trim_{q\alpha}(\mathscr{Y}\alpha)$
and so the refinement is regular. Hence $\mathscr{Y}\alpha$ is a
trim $P$-partition which regularly consolidates $\mathscr{Y}$ via
$\alpha$, as required.

Uniqueness is now immediate from Proposition~\ref{Prop:regrefprops}\ref{enu:refine1-3}.
\end{proof}
\begin{remark}
The extension from $\mathscr{Z}$ to $\mathscr{Y}\alpha$ in the proof
is necessary. For let $Q$ be any countable PO system not satisfying
the ascending chain condition such that $Q^{d}=\emptyset$. Then there
is a countable atomless Boolean ring whose Stone space $W$ admits
a trim $Q$-partition $\mathscr{Y}$, and $\mathscr{Y}$ will not
be complete (\cite[Theorem~5.1]{Apps-Stone}). However there is a
morphism $\alpha$ from $Q$ to $P=\{p\}$, where $p<p$, and the
trim $P$-partition of $W$ is just $\mathscr{Y}\alpha=\{W\}$.
\end{remark}

In the statement of the following Theorem, if $Q,P_{1},P_{2}$ are
PO systems and $\alpha_{i}\colon Q\rightarrow P_{i}$ $(i=1,2)$ are
surjective morphisms, we will say that the pairs $(P_{i},\alpha_{i})$,
which we term \emph{$Q$-pairs}, are equivalent if there is a PO system
isomorphism $\beta\colon P_{1}\rightarrow P_{2}$ such that $q\alpha_{2}=q\alpha_{1}\beta$
for all $q\in Q$.
\begin{theorem}
\label{Prop:morphism equiv classes}Let $Q$ be a PO system and $\mathscr{Y}$
a trim $Q$-partition of the primitive $\omega$-Stone space $W$.
Then there is a 1--1 correspondence between (i) trim partitions of
$W$ that regularly consolidate $\mathscr{Y}$, and (ii) equivalence
classes of $Q$-pairs $(P,\alpha)$. 
\end{theorem}

\begin{proof}
Proposition~\ref{Prop:regrefprops}\ref{enu:refine1-1} and Theorem~\ref{Prop:trim image under morphism}
establish each-way maps between trim partitions of $W$ that regularly
consolidate $\mathscr{Y}$ and equivalence classes of $Q$-pairs $(P,\alpha)$:
for if $Q$-pairs $(P_{1},\alpha_{1})$ and $(P_{2},\alpha_{2})$
are equivalent, then $\mathscr{Y}\alpha_{1}=\mathscr{Y}\alpha_{2}$
(with the partition relabelling given by the isomorphism between $P_{1}$
and $P_{2}$). By the uniqueness of $\mathscr{Y}\alpha$ these each-way
maps are inverse to each other. Hence the correspondence is 1--1,
as required. 
\end{proof}

\subsection{Structure diagrams and the canonical trim partition of a primitive
space}

We show that there is a canonical trim partition of a primitive space,
which regularly consolidates all other trim partitions of the space.
\begin{definition}[\emph{Structure diagram of a Boolean space}]
Let $W$ be a primitive $\omega$-Stone space, let $R=\Co(W)$, and
let $\PI(R)$ denote the set of all the PI\ elements of~$R$. Then:
\begin{enumerate}
\item the \emph{structure diagram of $W$ }is the PO system $\mathscr{S}(W)=\{[A]\mid A\in\PI(R)\}$,
with the relation $[A]<[B]$ iff $[A]\times[B]\cong[A]$;
\item the \emph{canonical partition }of $W$ is given by $\mathscr{X}(W)=\{X_{p}\mid p\in\mathscr{S}(W)\}^{*}$,
where $X_{p}$ is the set of points with a neighbourhood base of sets
isomorphic to $p$, for $p\in\mathscr{S}(W)$.
\end{enumerate}
\end{definition}

\begin{remark}
\label{rem: order reversal}$\mathscr{S}(W)$ is the same as the structure
diagram defined by Hanf (see~\cite[Construction~5.3]{Hanf}), but
with the order relation reversed. The usual order reversal issues
arise here between Boolean rings and their Stone spaces. Hanf along
with many other authors adopts the convention for PI Boolean algebras
whereby the unit element is of maximum type in $P$, which is a natural
choice when considering (for example) isomorphism classes of ideals
$(A)$ for $A\in R$. However, when working with partitions of the
Stone space, it seems more intuitive that finite partition elements
$X_{p}$ (for certain $p\in P^{d}$) should correspond to \emph{minimum
}elements $p\in P$. We have adopted the latter convention in our
definitions, consistent with the order direction in Ziegler's definition
of a good partition (\cite[Part~II,~1.C,D]{FlumZiegler}), with the
benefit for complete partitions that the open subsets of $W$ of the
form $X_{Q}$ correspond to upper subsets $Q$ of $P$, which are
the open subsets under the usual Alexandroff topology on $P$.
\end{remark}

The canonical partition $\mathscr{X}(W)$ of a primitive space $W$
is trim:
\begin{theorem}[Canonical trim partition]
\label{Thm:canonical trim}Let $W$ be a primitive $\omega$-Stone
space. Then $\mathscr{X}(W)$ is a trim $\mathscr{S}(W)$-partition
of $W$, $\Trim(\mathscr{X}(W))=\PI(R)$, and $A\in R$ is $p$-trim
iff $[A]=p$, where $R=\Co(W)$.
\end{theorem}

\begin{proof}
We claim first that if $A\in\PI(R)$ then $A\cap X_{[A]}\neq\emptyset$.
For let $\{B_{n}\mid n\geqslant1\}$ be an enumeration of $R$. Using
the fact that $A$ is PI, and writing $A_{n}=(A_{n}\cap B_{n})\dotplus(A_{n}-B_{n})$,
we can construct $A=A_{0}\supseteq A_{1}\supseteq A_{2}\supseteq\cdots$
such that $[A_{n}]=[A]$ and either $A_{n}\subseteq B_{n}$ or $A_{n}\cap B_{n}=\emptyset$.
By compactness and our choice of $A_{n}$, $\bigcap_{n\geqslant1}A_{n}$
is a singleton $\{w_{A}\}$, say, and now $w_{A}\in X_{[A]}$.

We show next that if $A\in\PI(R)$ then $A$ is $[A]$-trim. Choose
$w\in A\cap X_{q}$, where $q=[A]$. If $w\in X_{p}^{\prime}$, find
$x\in(A-\{w\})\cap X_{p}$, and find $B\in R$ such that $x\in B\subseteq A-\{w\}$
and $[B]=p$. If $q\neq p$, then $[A-B]=[A]$ and $q<p$. If however
$q=p$, then we can choose $C$ such that $w\in C\subseteq A-B$ and
$[C]=q$, and $A$ contains the compact open subset $B\dotplus C$
with $[B]=[C]=q$. It follows from~\cite[Lemma~3.14]{Apps-Stone}
that $(A)\cong(B)\times(C)$ and so $q<q$. The same argument shows
that if $q\nless q$ then $|A\cap X_{q}|=1$.

Conversely, if $q<p$, then we can write $A=B\dotplus C$, where $[B]=q$
and $[C]=p$, with $w\in B$ if $p=q$, and so $(A-\{w\})\cap X_{p}\neq\emptyset$.
Hence $A$ is $q$-trim.

If now $q\in\mathscr{S}(W)$ and $w\in X_{q}$, apply the above to
a neighbourhood base of $w$ of $q$-trim sets to obtain $w\in X_{p}^{\prime}$
iff $q<p$, so that $\mathscr{X}(W)$ is a $\mathscr{S}(W)$-partition
of~$W$.

Finally, if $A$ is $p$-trim then $A\in\PI(R)$ by~TP6 and so $[A]=p$.

The trim axioms now follow immediately: T1 as the PI sets generate
$R$, and T2 and T3 from the definition of $X_{p}$.
\end{proof}
\begin{remark}
This result is implicit in the proof of Apps~\cite[Theorem~3.12]{Apps-Stone},
which uses the duality between trim partitions and Hanf's ``structure
functions'', together with the fact (\cite[Theorem~5.4]{Hanf}) that
the map $u\colon\PI(R)\rightarrow\widetilde{\mathscr{S}(W)}:A\mapsto[A]$
``structures'' $R$ if $R$ is primitive.
\end{remark}

The structure diagram $\mathscr{S}(W)$ is known to have the following
properties in the case where $W$ is a compact PI primitive Boolean
space (interpreting these in the language of Boolean spaces rather
than Boolean rings):
\begin{theorem}
\label{Hanf}Let $W$ be a compact PI primitive $\omega$-Stone space.
\begin{enumerate}
\item (Hanf~\cite[Theorem~5.5]{Hanf}) $W$ is determined up to homeomorphism
by the isomorphism class of $\mathscr{S}(W)$;
\item \label{enu: simple POs}(e.g.\ see Pierce~\cite[section~3.8]{PierceMonk})
$\mathscr{S}(W)$ is a simple PO system. 
\end{enumerate}
\end{theorem}

Theorem~\ref{Hanf}\ref{enu: simple POs} and the associated result
of Williams~\cite[Corollary~10]{Williams} can be extended to general
primitive spaces:
\begin{theorem}
\label{Thm:trim partitions} Let $W$ be a primitive $\omega$-Stone
space. Then $\mathscr{S}(W)$ is a simple PO system.
\end{theorem}

\begin{proof}
Let $\mathscr{X}=\mathscr{X}(W)=\{X_{p}\mid p\in P\}^{*}$, where
$P=\mathscr{S}(W)$, and let $R=\Co(W)$. It is enough to show that
any surjective morphism $\alpha\colon P\rightarrow Q$ is an isomorphism. 

By Proposition~\ref{Prop:trim image under morphism}, there is a
trim $Q$-partition $\mathscr{Y}=\mathscr{X}\alpha=\{Y_{q}\mid q\in Q\}^{*}$
of~$W$ such that $\mathscr{X}$ regularly refines $\mathscr{Y}$
via $\alpha$. Suppose that $p,r\in P$ and $q\in Q$ are such that
$q=p\alpha=r\alpha$. Choose $A\in\Trim_{p}(\mathscr{X})$ and $B\in\Trim_{r}(\mathscr{X})$;
by Proposition~\ref{Prop:regref} $A,B\in\Trim_{q}(\mathscr{Y})$.
Hence by~TP6 $A$ and $B$ are homeomorphic, so $p=[A]=[B]=r$,
and $\alpha$ is an isomorphism.
\end{proof}
\begin{corollary}
\label{Cor:morphism to canonical}Let $W$ be a primitive $\omega$-Stone
space, $Q$ a PO system and $\mathscr{Y}$ a trim $Q$-partition of
$W$. Then $\mathscr{S}(W)\cong s(Q)$, and $\mathscr{Y}\prec_{\beta}\mathscr{X}(W)$,
where $\beta$ is the (unique) surjective morphism from $Q$ to $\mathscr{S}(W)$.
\end{corollary}

\begin{proof}
Let $\mathscr{X}(W)=\{X_{p}\mid p\in P\}^{*}$, where $P=\mathscr{S}(W)$,
and let $\mathscr{Y}=\{Y_{q}\mid q\in Q\}^{*}$. 

For $q\in Q$, let $q\beta\in\mathscr{S}(W)$ be the isomorphism type
of the $q$-trim subsets of $W$ (using~TP6), so that $\Trim_{q}(\mathscr{Y})\subseteq\Trim_{q\beta}(\mathscr{X}(W))$,
as PI sets and trim sets are the same in $\mathscr{X}(W)$. Then $Y_{q}\subseteq X_{q\beta}$,
as elements of $Y_{q}$ have a neighbourhood base of $q$-trim sets,
and so $\mathscr{Y}$ regularly refines $\mathscr{X}(W)$ via $\beta$,
and $\beta$ is a surjective morphism (Proposition~\ref{Prop:regrefprops}\ref{enu:refine1-1}).
The first statement follows immediately from Theorem~\ref{Thm:trim partitions}
and Theorem~\ref{Thm:simple image}.
\end{proof}
\begin{remark}
So any trim partition of a primitive $\omega$-Stone space $W$ is
a regular refinement of the canonical trim partition (cf~Williams~\cite[Theorem~9]{Williams}),
and the PO system $\mathscr{S}(W)$ can be determined from the structure
of any trim partition of $W$. 
\end{remark}

We can further characterise the canonical partition of a primitive
space:
\begin{theorem}
\label{Thm:equiv simple}Let $P$ be a countable PO system and $\mathscr{Y}$
a trim $P$-partition of the (primitive) $\omega$-Stone space $W$.
Then the following are equivalent:
\begin{enumerate}
\item \label{enu: equiv1}$\mathscr{Y}=\mathscr{X}(W)$ (relabelling if
required);
\item \label{enu:equiv2}$P\cong\mathscr{S}(W)$;
\item \label{enu: equiv3}$P$ is simple.
\end{enumerate}
\end{theorem}

\begin{proof}
\ref{enu: equiv1}~$\Rightarrow$~\ref{enu:equiv2}: immediate from
the definition of $\mathscr{X}(W)$.

\ref{enu:equiv2}~$\Rightarrow$~\ref{enu: equiv3}: immediate from
Theorem~\ref{Thm:trim partitions}.

\ref{enu: equiv3}~$\Rightarrow$~\ref{enu: equiv1}: by Corollary~\ref{Cor:morphism to canonical},
there is a surjective morphism $\alpha\colon P\rightarrow\mathscr{S}(W)$
such that $\mathscr{Y}\prec_{\alpha}\mathscr{X}(W)$. But $P$ is
simple, so $\alpha$ is an isomorphism. Hence by Proposition~\ref{Prop:regrefprops}\ref{enu:refine1-5}
$\mathscr{Y}=\mathscr{X}(W)$ after relabelling.
\end{proof}

\subsection{Extended PO systems and regular refinements of a partition}

We have seen in Theorem~\ref{Prop:morphism equiv classes} that regular
consolidations of a trim partition correspond to surjective morphisms
of the underlying PO system. We now address the question of when a
trim $P$-partition of a primitive space can be regularly refined
to a trim $Q$-partition. For this, we will need to consider the compactness
structure and size of the finite elements of the $P$-partition, which
cannot generally be deduced from $P$ itself, and the next definitions
provide a convenient way of describing these. We write $P_{\Delta}$
for $\{p\in P\mid\{p\}\text{ has a finite foundation}\}$.
\begin{definition}[\emph{Extended PO systems}]
\label{def:extended POS}

An \emph{extended PO system }is a triple $(P,L,f)$, where $P$ is
a PO system, $L$ is a lower subset of $P_{\Delta}$, and $f\colon L_{\min}^{d}\rightarrow\mathbb{N}_{+}$,
where $L_{\min}^{d}=L_{\min}\cap P^{d}$. 

Let $W$ be a primitive Stone space and $(P,L,f)$ an extended PO
system. We will say that a trim $P$-partition $\mathscr{Z}=\{Z_{p}\mid p\in P\}{}^{*}$
of $W$ is a trim \emph{$(P,L,f)$-partition of $W$ }if $\overline{Z_{p}}$
is compact precisely when $p\in L$, and $|Z_{p}|=f(p)$ if $p\in L_{\min}^{d}$;
and that $\mathscr{Z}$ is a \emph{bounded trim partition, }or \emph{a
bounded trim $(P,L,f)$-partition, }of $W$ if in addition $Z_{L}$
is relatively compact (i.e.\ $\overline{Z_{L}}$ is compact). 

We will say that $W$ is \emph{bounded }if $\mathscr{X}(W)=\{X_{p}\mid p\in\mathscr{S}(W)\}^{*}$
is a bounded $(\mathscr{S}(W),L_{W},f_{W})$-partition of $W$, where
$L_{W}=\{p\in\mathscr{S}(W)\mid\overline{X_{p}}\text{ is compact}\}$
and $f_{W}(p)=|X_{p}|$ for $p\in(L_{W}){}_{\min}^{d}$.

Two extended PO systems $(P,L,f)$ and $(Q,M,g)$ are isomorphic,
and we write $(P,L,f)\cong(Q,M,g)$, if there is an isomorphism $\alpha\colon P\rightarrow Q$
such that $L\alpha=M$ and $g(p\alpha)=f(p)$ for all $p\in L_{\min}^{d}$
(noting that if $L\alpha=M$ then $L_{\min}^{d}\alpha=M_{\min}^{d}$).
\end{definition}

\begin{remark}
The definitions are consistent, as if $\overline{Z_{p}}$ is compact
then $p\in P_{\Delta}$ by~\cite[Proposition~2.16(iii)]{Apps-Stone},
so that $L$ is indeed a lower subset of $P_{\Delta}$; and $f$ is
well-defined because $Z_{p}$ is finite iff $p\in L_{\min}^{d}$ (as
$Z_{p}$ is closed for $p\in P_{\min}$, using~T1 and~T2).

We will need to use the following Lemma twice:
\end{remark}

\begin{lemma}
\label{Lemma:bounded FF}If $P$ is a countable PO system, $\{Z_{p}\mid p\in P\}{}^{*}$
is a trim $P$-partition of the primitive $\omega$-Stone space $W$,
and $N$ is a subset of $P$ such that $\overline{Z_{N}}$ is compact,
then $N$ has a finite foundation.
\end{lemma}

\begin{proof}
Find a compact open $A$ such that $Z_{N}\subseteq A$. By~TP5 we
can express $A$ as a union of $p$-trim sets for $p\in T(A)_{\min}$;
and it follows that $T(A)_{\min}\cap N_{\downarrow}$ is a finite
foundation for $N$.
\end{proof}
We now consider the types of partition that can regularly refine a
given partition: specifically, for which PO systems $Q$ can a given
bounded trim $(P,L,f)$-partition be regularly refined by a trim $(Q,M,g)$-partition
for some choice of $M$ and $g$? As above, there must be a surjective
morphism $\alpha\colon Q\rightarrow P$ for this possibility even
to arise.
\begin{proposition}
\label{Prop:morphism and PLf}If $\alpha\colon Q\rightarrow P$ is
a surjective morphism between the PO systems $P$ and $Q$, and $\mathscr{Y}$
is a bounded trim $(Q,M,g)$-partition of $W$, then $\mathscr{Y}\alpha$
is a bounded trim $(P,L,f)$-partition of $W$, where $L=\{p\in P\mid Q(p)\subseteq M\}$
and $f(p)=\sum_{q\in Q(p)}g(q)$ for $p\in L_{\min}^{d}$, writing
$Q(p)=\{p\}\alpha^{-1}$.
\end{proposition}

\begin{proof}
Let $\mathscr{Y}\alpha=\{U_{p}\mid p\in P\}^{*}$, and let $Z_{p}=\bigcup_{q\in Q(p)}Y_{q}$.
Then $\overline{Z_{p}}$ is compact iff $Q(p)\subseteq M$, as $\overline{Y_{M}}$
is compact; and $|Z_{p}|=\sum_{q\in Q(p)}g(q)$ for $p\in L_{\min}^{d}$,
as if $p\in L_{\min}^{d}$ then $Q(p)\subseteq M_{\min}^{d}$. These
results then also hold for $\mathscr{Y}\alpha$, since $U_{p}\subseteq\overline{Z_{p}}$
(by considering the construction in the proof of Proposition~\ref{Prop:trim image under morphism}),
and $\mathscr{Y}\alpha$ is bounded as $\overline{Z_{L}}\subseteq\overline{Y_{M}}$.
\end{proof}
We will also need the following Lemma.
\begin{lemma}
\label{Prop: direct sum}
\begin{enumerate}
\item \label{enu:sum1}A countable Boolean ring is primitive iff it is the
direct sum of countable primitive Boolean algebras;
\item \label{enu:sum2}let $P$ be a PO system and $P_{n}$ an upper subset
of $P$ for each $n$ such that $P=\bigcup_{n\geqslant1}P_{n}$. Suppose
$W=W_{1}\dotplus W_{2}\dotplus\ldots$, where each $W_{n}$ is a clopen
subset of the $\omega$-Stone space $W$ and admits a trim $P_{n}$-partition
$\mathscr{Z}_{n}=\{Z_{np}\mid p\in P_{n}\}^{*}$. Then $\mathscr{Y}=\{Y_{p}\mid p\in P\}^{*}$
is a trim $P$-partition of $W$, where $Y_{p}=\bigcup_{n\geqslant1}\{Z_{np}\mid p\in P_{n}\}$.
\end{enumerate}
\end{lemma}

\begin{proof}
\ref{enu:sum1} If $R$ is a countable primitive Boolean ring without
a $1$, then by a routine argument we can find disjoint $\{A_{n}\in R\mid n\geqslant1\}$
that generate $R$. It follows easily that $(A_{n})$ is a primitive
Boolean algebra and that $R=\bigoplus{}_{n\geqslant1}(A_{n})$. Conversely,
if $R=\bigoplus_{n\geqslant1}R{}_{n}$ with each $R_{n}$ a primitive
Boolean algebra, then each $A\in R$ has the form $A=A_{1}\dotplus\cdots\dotplus A_{N}$
with $A_{n}\in R_{n}$, for some $N$, and so is the sum of PI elements.

\ref{enu:sum2} Let $Z_{n}=\bigcup_{p\in P_{n}}Z_{np}$ and $Y=\bigcup_{p\in P}Y_{p}=\bigcup_{n\geqslant1}Z_{n}$.
Then $Y_{p}^{\prime}\cap Y\cap W_{n}=Z_{np}^{\prime}\cap Z_{n}=\bigcup_{q\in P_{n}}\{Z_{nq}\mid q<p\}$.
So $Y_{p}^{\prime}\cap Y=\bigcup_{q\in P}\{Y_{q}\mid q<p\}$ and $\mathscr{Y}$
is a $P$-partition of $W$. Moreover, if $A\subseteq W_{n}$ and
$p\in P_{n}$, then $A\in\Trim_{p}(\mathscr{Y})$ iff $A\in\Trim_{p}(\mathscr{Z}_{n})$,
as $P_{n}$ is an upper subset of $P$. The trim conditions T1 to
T3 now follow easily for $W$ as they are true in each $W_{n}$.
\end{proof}
\begin{theorem}
\label{Thm:refinement condition}Let $(P,L,f)$ be a countable extended
PO system, $\mathscr{Z}$ a bounded trim $(P,L,f)$-partition of the
$\omega$-Stone space $W$, $Q$ a countable PO system, and $\alpha\colon Q\rightarrow P$
a surjective morphism. Then $\mathscr{Z}$ can be regularly refined
to a trim $Q$-partition of $W$ iff both of the following hold:
\begin{enumerate}
\item \label{enu:cond2}$L\alpha^{-1}$ has a finite foundation; 
\item \label{enu:cond3}$|Q(p)|\leqslant f(p)$ for $p\in L_{\min}^{d}$,
where $Q(p)=\{q\in Q\mid q\alpha=p\}$.
\end{enumerate}
\end{theorem}

\begin{proof}
Let $R=\Co(W)$ and $\mathscr{Z}=\{Z_{p}\mid p\in P\}^{*}$. Choose
$B\in R$ such that $Z_{L}\subseteq B$, so that $L\subseteq T(B)$,
where $T\colon R\rightarrow2^{P}$ is the type function with respect
to $\mathscr{Z}$. Use TP5 to write $B$ as a finite disjoint union
of $q$-trim sets as $q$ runs through $T(B)_{\min}$. We can now
discard any $q$-trim sets $A$ where $q\notin L$, as then $L\cap T(A)=\emptyset$,
to obtain $B$ as a union of $q$-trim sets for certain $q\in L$,
so that $T(B)=L_{\uparrow}$.

Suppose first that $\mathscr{Z}$ can be regularly refined to a trim
$Q$-partition $\{Y_{q}\mid q\in Q\}^{*}$ of $W$. Then $Y_{L\alpha^{-1}}\subseteq Z_{L}\subseteq B$,
so by Lemma~\ref{Lemma:bounded FF} $L\alpha^{-1}$ has a finite
foundation. Condition~\ref{enu:cond3} is immediate by counting. 

Suppose now that conditions~\ref{enu:cond2}~and~\ref{enu:cond3}
hold. We consider first the case when $W$ is compact. If $p\in P_{\min}^{d}$
then $Q(p)\subseteq Q_{\min}^{d}$, so we can define a function $g\colon Q(p)\rightarrow\mathbb{N}_{+}$
such that $\sum_{q\in Q(p)}g(q)=|Z_{p}|$, and extend this to a function
$g\colon Q_{\min}^{d}\rightarrow\mathbb{N}_{+}$, letting $g(q)=1$
if $q\alpha\notin P_{\min}^{d}$.

By~\cite[Theorem~5.5]{Apps-Stone}, noting that $L=P$ and so $Q$
has a finite foundation as $Q=L\alpha^{-1}$, there is a compact $\omega$-Stone
space $Y$ with a trim $Q$-partition $\mathscr{Y}=\{Y_{q}\mid q\in Q\}^{*}$
such that $|Y_{q}|=g(q)$ for $q\in Q_{\min}^{d}$. Applying the construction
of Proposition~\ref{Prop:trim image under morphism}, we obtain a
trim $P$-partition $\mathscr{U}=\{U_{r}\mid r\in P\}^{*}$ of $Y$,
with $|U_{p}|=|Z_{p}|$ if $p\in P^{d}$ (as the final step of the
proof of Proposition~\ref{Prop:trim image under morphism} will not
expand any finite sets), such that $\mathscr{Y}$ regularly refines
$\mathscr{U}$. By~\cite[Theorem~6.1]{Apps-Stone} there is a $P$-homeomorphism
$\beta$ from $Y$ to $W$, and so $\{Y_{q}\beta\mid q\in Q\}^{*}$
is a trim $Q$-partition of $W$ that regularly refines $\mathscr{Z}$,
as required.

We turn now to the case of a general $W$. 

\textbf{Step 1}: 

Let $M$ be a finite foundation for $L\alpha^{-1}$. Now $M_{\uparrow}$
has a finite foundation, $\alpha$ restricts to a surjective morphism
from $M_{\uparrow}$ to $L_{\uparrow}$, and if $Z_{p}$ is finite
then $p\in L$ and $|Q(p)|\leqslant|Z_{p}|$. Applying the compact
case (taking $(W,P,Q)=(B,L_{\uparrow},M_{\uparrow})$), we obtain
a trim $M_{\uparrow}$-partition of $B$ that regularly refines $\mathscr{Z}|_{B}$,
writing $\mathscr{Z}|_{D}$ for the restriction of the partition $\mathscr{Z}$
to a subset $D\subseteq W$.

\textbf{Step 2}: Using Lemma~\ref{Prop: direct sum}\ref{enu:sum1}
and TP5, split $W-B$ into a finite or countable disjoint union of
trim sets. Grouping the partition into trim sets of the same type,
we can write $R=(B)\oplus\left\{ \bigoplus_{p\in U}\bigoplus_{n\leqslant N_{p}}(C_{p,n})\right\} $,
where $U=P-L$, $C_{p,n}$ is $p$-trim and $0\leqslant N_{p}\leqslant\infty$
$(p\in U)$. 

We can assume that $N_{r}=\infty$ if $r\in U$. For let $k\colon P\rightarrow\mathbb{N}$
be an injective ``labelling'' of $P$. We recall that if $s>p$
($s,p\in P$), $A$ is $s$-trim and $B$ is $p$-trim then $[B]\cong[B]\times[A]$.
Hence we can express each $C_{p,n}$ as a finite union of $s$-trim
sets as $s$ ranges through the finite set $\{s\in P\mid s\geqslant p,k(s)\leqslant k(p)+n\}\subseteq U$,
and regroup into trim sets of the same type. As $\overline{Z_{r}}$
is not compact for $r\in U$, there are infinitely many of the original
$C_{p,n}$ such that $Z_{r}\cap C_{p,n}\neq\emptyset$, i.e.\ such
that $r\geqslant p$, so either there are infinitely many such $p$,
or for some $p$ there are infinitely many such $n$. Hence there
are infinitely many of the revised $C_{p,n}$ such that $p=r$, so
that $N_{r}=\infty$.

\textbf{Step 3}: for each $p\in U$, let $h_{p}\colon\mathbb{N}\rightarrow Q(p)$
be a surjection such that $\{n\in\mathbb{N}\mid h_{p}(n)=q\}$ is
infinite for all $q\in Q(p)$. Now for $q\in Q(p)$, $\alpha$ restricts
to a surjective morphism from $q_{\uparrow}$ to $p_{\uparrow}$,
and if $p\in P^{d}$ then $Q(p)\cap q_{\uparrow}=\{q\}$. We can therefore
apply the compact case (taking $(W,P,Q)=(C_{p,n},p_{\uparrow},h_{p}(n){}_{\uparrow})$),
to obtain a trim $h_{p}(n){}_{\uparrow}$-partition of $C_{p,n}$
that regularly refines $\mathscr{Z}|_{C_{p,n}}$, for each $n$.

\textbf{Step 4}: we claim that $Q=M_{\uparrow}\cup\left\{ \bigcup\{h_{p}(n){}_{\uparrow}\mid p\in U,n\geqslant1\}\right\} $.
Choose $q\in Q$; if $q\alpha\in L$, then $q\in L\alpha^{-1}\subseteq M_{\uparrow}$.
If instead $q\alpha\in U$, then we can find infinitely many $n\in\mathbb{N}$
such that $q=h_{q\alpha}(n)$.

\textbf{Step 5}: finally, we can apply Lemma~\ref{Prop: direct sum}\ref{enu:sum2}
to obtain a trim $Q$-partition of $W$ which refines $\mathscr{Z}$,
and the refinement is regular since its restriction to $B$ and each
$C_{p,n}$ is regular.
\end{proof}
\begin{remark}
Letting $\{V_{q}\mid q\in Q\}^{*}$ denote the trim $Q$-partition
constructed above, we see that $\overline{V_{q}}$ is compact iff
$q\alpha\in L$, and so the constructed $Q$-partition is bounded.
However, a different choice for the functions $h_{p}$ might yield
some relatively compact elements $V_{q}$ for $q\in U$, whose union
was not relatively compact.

If $\mathscr{Z}$ is not bounded, condition~\ref{enu:cond3} clearly
remains necessary. The following weaker condition than~\ref{enu:cond2}
is also necessary, but it is not obvious whether or not this condition
is sufficient:

(c) if $M$ is a lower subset of $P$ such that $\overline{Z_{M}}$
is compact, then $M\alpha^{-1}$ has a finite foundation. 
\end{remark}

\subsection{Some classification results}

The existing literature (e.g.\ \cite[Theorem~3.8.3]{PierceMonk})
establishes a bijection between isomorphism classes of primitive PI
Boolean algebras and those of countable simple diagrams (PO systems
with a maximum element). Theorem~\ref{Thm:equiv simple}, together
with the various uniqueness theorems in~\cite{Apps-Stone}, allow
us to generalise this as follows, using the notation of Definition~\ref{def:extended POS}.
We write $\boldsymbol{EPS}$ for the set of isomorphism classes of
countable extended PO systems $(P,L,f)$ such that $P$ is simple,
and $[E]$ for the isomorphism class of a structure $E$.
\begin{theorem}
The map $\theta\colon W\mapsto[(\mathscr{S}(W),L_{W},f_{W})]$ yields
a surjection from the class of primitive $\omega$-Stone spaces to
$\boldsymbol{EPS}$, and induces the following bijections:
\begin{enumerate}
\item \label{enu:bij1}between the homeomorphism classes of compact primitive
$\omega$-Stone spaces and isomorphism classes $[(P,f)]$, where $P$
is a countable simple PO system with a finite foundation and $f\colon P_{\min}^{d}\rightarrow\mathbb{N}_{+}$;
\item \label{enu:bij2}between the homeomorphism classes of finitary $\omega$-Stone
spaces (primitive spaces $W$ with $\mathscr{S}(W)$ finite) and classes
$[(P,L,f)]$, where $P$ is a finite simple PO system, $L$ is a lower
subset of $P$ and $f\colon L_{\min}^{d}\rightarrow\mathbb{N}_{+}$;
\item \label{enu:bij3}between the homeomorphism classes of bounded $\omega$-Stone
spaces and isomorphism classes $[(P,L,f)]$, where $P$ is a countable
simple PO system, $L$ is a lower subset of $P$ having a finite foundation,
and $f\colon L_{\min}^{d}\rightarrow\mathbb{N}_{+}$.
\end{enumerate}
\end{theorem}

\begin{proof}
By Theorem~\ref{Thm:trim partitions}, $\theta(W)\in\boldsymbol{EPS}$
for primitive $W$. Also, if $[(P,L,f)]\in\boldsymbol{EPS}$, then
by~\cite[Theorem~5.5]{Apps-Stone} (taking $L^{u}=L$) we can find
an $\omega$-Stone space $W$ that admits a trim $(P,L,f)$-partition
$\mathscr{Y}$; but $P$ is simple, so applying Theorem~\ref{Thm:equiv simple}
we see that $\mathscr{Y}$ is just a relabelling of $\mathscr{X}(W)$,
and so $[(P,L,f)]=\theta(W)$ and $\theta$ is surjective.

Statement~\ref{enu:bij1} is immediate from~\cite[Theorem~6.1]{Apps-Stone},
taking $L=P$, and statement~\ref{enu:bij2} from~\cite[Corollary~6.2]{Apps-Stone}.

For~\ref{enu:bij3}, we observe that if $W$ is a bounded $\omega$-Stone
space, we can find a compact open subset $B$ of $W$ such that $T_{\mathscr{X}(W)}(B)=(L_{W}){}_{\uparrow}$,
by using the first step in the proof of Theorem~\ref{Thm:refinement condition},
and so by~\cite[Proposition~2.16(i)]{Apps-Stone} $L_{W}$ has a
finite foundation. If now $W$ and $X$ are bounded $\omega$-Stone
spaces such that $\theta(W)=\theta(X)=[(P,L,f)]$, say, where $L$
has a finite foundation, then $W$ and $X$ each admit a bounded trim
$(P,L,f)$-partition, and so are themselves homeomorphic by~\cite[Theorem~6.1]{Apps-Stone},
as required.
\end{proof}
\begin{remark}
Statement~\ref{enu:bij1} is also implicit in Hansoul~\cite{Hansoul2},
and in the exposition in~\cite[section~3]{PierceMonk} which views
isomorphism classes of primitive Boolean algebras as a monoid.

Statement~\ref{enu:bij3} breaks down however once we allow unbounded
$\omega$-Stone spaces, as the following example shows.
\end{remark}

\begin{example}
\label{exa:Iso P, non homeo space}Let $P=\{a_{n},b_{n},c\mid n\geqslant0\}$
with $P^{d}=\{a_{n},c\mid n\geqslant0\}$, and with the order relation
generated by the following for $n\geqslant0$:

\[
\begin{aligned}a_{n}>a_{n+1} &  &  &  &  & a_{n}>b_{n}>c.\end{aligned}
\]

It can be shown that $P$ is simple, with $P=P_{\Delta}$. Let $L=\{b_{n},c\mid n\geqslant0\}$,
which has a finite foundation $\{c\}$, and let $f(c)=1$. Applying~\cite[Theorem~5.5]{Apps-Stone}
with $Q=L$, taking $L^{u}=\emptyset$ and $L^{u}=L$ in turn, we
obtain $\omega$-Stone spaces $X$ and $Y$ with $(P,L,f)$-partitions
$\{X_{p}\mid p\in P\}^{*}$ and $\{Y_{p}\mid p\in P\}^{*}$ respectively
such that $\overline{X_{L}}$ is compact but $\overline{Y_{L}}$ is
not. Relabelling if necessary so that $P=\mathscr{S}(X)=\mathscr{S}(Y)$,
we have $L=L_{W}=L_{X}$, and so $X$ and $Y$ are not homeomorphic.
\end{example}

\begin{question}
For a primitive $\omega$-Stone space $W$, we have seen that information
on which $\overline{X_{p}}$ are compact for $p\in P=\mathscr{S}(W)$
may not be sufficient to determine the homeomorphism class of $W$.
Suppose instead that we know which subsets $Q$ of $P$ give rise
to compact sets $\overline{X_{Q}}$. That is, we consider $[P,\mathscr{L},f]$-partitions,
where $\mathscr{L}$ is a lower subset of $2^{P}$ such that $\overline{X_{Q}}$
is compact precisely when $Q\in\mathscr{L}$, and $f\colon\mathscr{L}_{\min}^{d}\rightarrow\mathbb{N}_{+}$,
where $\mathscr{L}_{\min}^{d}=\{p\in P_{\min}^{d}\mid\{p\}\in\mathscr{L}\}$.
Are two $\omega$-Stone spaces admitting a trim $[P,\mathscr{L},f]$-partition
necessarily homeomorphic?
\end{question}

\subsection{The class of bounded trim partitions of a fixed Boolean space}

The preceding results lead to a natural quasi-ordering on equivalence
classes of bounded trim partitions of a fixed space, which we will
see is isomorphic to a set of isomorphism classes of extended PO systems.
\begin{definition}
\label{def:equiv classes}For an $\omega$-Stone space $W$, let $\mathscr{T}(W)$
denote the set of equivalence classes of bounded trim partitions of
$W$, with two trim partitions deemed equivalent if they are the same
up to homeomorphism and relabelling. For $[\mathscr{Y}],[\mathscr{Z}]\in\mathscr{T}(W)$,
write $[\mathscr{Y}]\prec[\mathscr{Z}]$ if some (any) element of
$[\mathscr{Y}]$ regularly refines some element of $[\mathscr{Z}]$.

Let $\mathscr{P}(W)$ denote the set of isomorphism classes of extended
PO systems $(P,L,f)$ for which there is a bounded trim $(P,L,f)$-partition
of $W$. 

For $[(P,L,f)],[(Q,M,g)]\in\mathscr{P}(W)$, write $[(Q,M,g)]\prec[(P,L,f)]$
if there is a surjective morphism $\alpha\colon Q\rightarrow P$ such
that $L=\{p\in P\mid Q(p)\subseteq M\}$ and $f(p)=\sum_{q\in Q(p)}g(q)$
for $p\in L_{\min}^{d}$, where $Q(p)=\{p\}\alpha^{-1}$, noting that
if this is true for some choice of $Q$ and $P$ within their isomorphism
classes then it is true for all such choices. 
\end{definition}

\begin{remark}
For fixed $W$, let $\mathscr{C}(W)$ denote the set of isomorphism
classes of countable PO systems $Q$ such that $s(Q)\cong\mathscr{S}(W)$.
By Corollary~\ref{Cor:morphism to canonical} and Theorem~\ref{Thm:refinement condition},
$[(Q,M,g)]\in\mathscr{P}(W)$ for some choice of $M$ and $g$ iff
$[Q]\in\mathscr{C}(W)$, $L_{W}\alpha_{Q}^{-1}$ has a finite foundation
and $|\{p\}\alpha_{Q}^{-1}|\leqslant f_{W}(p)$ for all $p\in(L_{W})_{\min}^{d}$,
where $\alpha_{Q}$ denotes the unique surjective morphism from $Q$
to $\mathscr{S}(W)$.
\end{remark}

\begin{theorem}
Let $W$ be a bounded $\omega$-Stone space. Then $(\mathscr{T}(W),\prec)$
and $(\mathscr{P}(W),\prec)$ are isomorphic to each other as quasi-ordered
systems, and have maximal elements $[\mathscr{X}(W)]$ and $[(\mathscr{S}(W),L_{W},f_{W})]$
respectively.
\end{theorem}

\begin{proof}
It follows from Proposition~\ref{Prop:regrefprops}\ref{enu:refine1-4}
that $\prec$ is transitive on $\mathscr{T}(W)$. Easy checks show
that the composition of two morphisms is a morphism, and that $\prec$
is transitive on $\mathscr{P}(W)$. 

Define a map $\theta\colon\mathscr{T}(W)\rightarrow\mathscr{P}(W)$
by mapping $[\mathscr{Y}]$ to $[(Q,M,g)]$ if $\mathscr{Y}$ is a
bounded trim $(Q,M,g)$-partition of $W$; this is well-defined as
homeomorphisms of partitions preserve $M$ and $g$, and relabelling
corresponds to a PO system isomorphism.

The map $\theta$ is clearly surjective. Conversely, if $\mathscr{Y}=\{Y_{q}\mid q\in Q\}^{*}$
is a bounded $(Q,M,g)$-partition of $W$, then $M$ has a finite
foundation by Lemma~\ref{Lemma:bounded FF}. Hence by~\cite[Theorem~6.1]{Apps-Stone},
all $(Q,M,g)$-partitions of $W$ are homeomorphic to each other.
Therefore $\theta$ is a bijection, as PO system isomorphisms correspond
to a relabelling of the partition. We must show that $\theta$ is
order-preserving.

If $\mathscr{Y}\prec_{\alpha}\mathscr{Z}$, then $\mathscr{Z}=\mathscr{Y}\alpha$
and $[\mathscr{Y}]\theta\prec[\mathscr{Z}]\theta$ by Proposition~\ref{Prop:morphism and PLf}. 

Conversely, suppose that $\mathscr{Y}$ is a bounded trim $(Q,M,g)$-partition
of $W$, and that $[(Q,M,g)]\prec[(P,L,f)]$ via the surjective morphism
$\alpha\colon Q\rightarrow P$. By Propositions~\ref{Prop:trim image under morphism}
and~\ref{Prop:morphism and PLf} $\mathscr{\mathscr{Y}\prec_{\alpha}Y}\alpha$
and $\mathscr{Y}\alpha$ is a bounded trim $(P,L,f)$-partition of
$W$. So if $[\mathscr{Y}]\theta\prec[\mathscr{Z}]\theta=[(P,L,f)]$
then $[\mathscr{Z}]\theta=[\mathscr{Y}\alpha]\theta$ and $[\mathscr{Y}]\prec[\mathscr{Y}\alpha]=[\mathscr{Z}]$,
as $\theta$ is injective. Therefore $\theta$ is an isomorphism of
quasi-ordered systems, as required.

Finally, the maximality of $[\mathscr{X}(W)]$ and $[(\mathscr{S}(W),L_{W},f_{W})]$
follows from Proposition~\ref{Cor:morphism to canonical}.
\end{proof}
\begin{remark}
If $W$ is a \emph{finitary }$\omega$-Stone space (i.e.\ where $\mathscr{S}(W)$
is finite) and we restrict $(\mathscr{T}(W),\prec)$ and $(\mathscr{P}(W),\prec)$
to partitions where the underlying PO systems are finite, then $\prec$
will be antisymmetric and $(\mathscr{T}(W),\prec)$ and $(\mathscr{P}(W),\prec)$
will be isomorphic as posets.

However, these quasi-orders are not in general antisymmetric:
\end{remark}

\begin{example}
Let $P_{j}$ ($j=1,2$) be the set of sequences $\{(m_{0},m_{1},m_{2},\ldots)\}$
such that $m_{n}\in\mathbb{N}$ for all $n$, $m_{0}\leqslant j$
and only finitely many $m_{n}>0$, with the pointwise order relation
and such that $p<p$ for all $p\in P_{j}$. Define $\alpha_{j}\colon P_{j}\rightarrow P_{3-j}$
by setting $(m_{0},m_{1},m_{2},\ldots)\alpha_{j}=(\min(m_{1},3-j),m_{2},m_{3},\ldots)$.
It can be readily verified that $\alpha_{j}$ ($j=1,2$) are surjective
morphisms. 

However, $P_{1}$ and $P_{2}$ are not isomorphic. To see this, let
$p=(1,0,0,\ldots)\in P_{1}$, which covers $\boldsymbol{0}$, the
minimum element of $P_{1}$. Let $Q=\{q\in P_{1}\mid p\nless q\}$;
then $Q=\{(0,m_{1},m_{2},\ldots)\}$, and so every element of $P_{1}-Q$
covers an element of $Q$. However, there is no element $p\in P_{2}$
covering $\boldsymbol{0}$ for which this is true. 

Finally, as $P_{1}$ and $P_{2}$ both have a finite foundation and
no discrete elements, we can find a (bounded) trim $P_{1}$-partition
$\mathscr{X}$ and a trim $P_{2}$-partition $\mathscr{Y}$ of the
Cantor set, with $[\mathscr{X}]\prec[\mathscr{Y}]$ and $[\mathscr{Y}]\prec[\mathscr{X}]$,
but $[\mathscr{X}]\neq[\mathscr{Y}]$; equivalently, $[(P_{1},P_{1},f)]\prec[(P_{2},P_{2},f)]$
and vice versa, where $f$ is the empty function, but $[(P_{1},P_{1},f)]\neq[(P_{2},P_{2},f)]$. 
\end{example}

\section{\label{sec:Complete-partitions}Complete partitions of primitive
spaces}

In this section we clarify the relationship between trim partitions,
which are not complete unless the underlying PO system satisfies the
ascending chain condition, and the (complete) \emph{rank diagrams
}introduced by Myers~\cite{Myers}. We define \emph{rank partitions},
as a generalisation of rank diagrams; define the \emph{ideal completion}
of a trim $P$-partition, whose underlying PO system is the ideal
completion of $P$; and show that rank partitions correspond precisely
to ideal completions of trim partitions.

\subsection{Rank diagrams and rank partitions}

First, we need some more definitions.
\begin{definition}[\emph{Separated and compact elements of a PO system}]

Let $P$ be a PO system. An element $p\in P$ is said to be \emph{separated
in $P$ }if it is not the supremum of a strictly increasing sequence
in $P$, and $p$ is \emph{compact in $P$} if for every ideal $J$
of $P$ such that $p\leqslant\sup_{P}J$, we have $p\in J$. For countable
$P$, $p$ is separated in $P$ if for every ideal $J$ of $P$ such
that $p=\sup_{P}J$, we have $p\in J$.

If $Q$ is another PO system such that $P\subseteq Q$, we say that
$P$ is \emph{weakly separated in~$Q$} if $p=\sup_{Q}J$ for some
$p\in P$ and ideal $J$ of $P$ implies that $p\in J$ (equivalently,
for countable $P$, if no $p\in P$ is the supremum in $Q$ of a strictly
increasing sequence in~$P$), and that $Q$ is \emph{$\omega$-complete
over $P$ }if every ascending sequence of elements of $P$ has a supremum
in $Q$.
\end{definition}

\begin{remark}
It is easy to see that:
\begin{enumerate}
\item If $p$ is compact in $P$ then $p$ is separated in $P$;
\item if $P$ is countable, $P\subseteq Q$ and all elements of $P$ are
separated in $Q$, then $P$ is weakly separated in~$Q$;
\item if $P$ is weakly separated in $P$ (e.g.\ if $P$ is countable and
all elements of $P$ are separated in $P$) and $P\subseteq Q$, then
$P$ is weakly separated in~$Q$, as if $p=\sup_{Q}J$ for some $p\in P$
and ideal $J$ of $P$ then $p=\sup_{P}J$.
\end{enumerate}
\end{remark}

\begin{example}
Let $P=\{r\}$, $Q=\mathbb{N}\cup\{r\}$, and $S=Q\cup\{\omega\}$
where $j<r$ in $Q$ and $j<\omega<r$ in $S$ for all $j\in\mathbb{N}$.
Then $r$ is separated in $P$ and $S$, but not in $Q$. We note
that $S=\id(Q)$, as defined below.
\end{example}

\begin{definition}[\emph{Semi-trim partitions and clean interior}]
Let $W$ be a Stone space, $P$ a PO system and $\mathscr{X}=\{X_{p}\mid p\in P\}^{*}$
a $P$-partition of $W$. We say that $\mathscr{X}$ is a \emph{semi-trim
}$P$-partition of $W$ if it satisfies trim axioms T1 and T2, with
T3 replaced by:
\begin{description}
\item [{ST3\label{ST3}}] Clean points are dense in $X_{p}$ for $p\in\widehat{P}$
\end{description}
and that $\mathscr{X}$ is \emph{a strongly semi-trim }$P$-partition
of $W$ if it satisfies T1 and T2, with T3 replaced by 
\begin{description}
\item [{SST3\label{SST3}}] All points in $X_{p}$ are clean for $p\in\widehat{P}$.
\end{description}
The \emph{clean interior }of a semi-trim partition $\mathscr{X}$,
denoted $\mathscr{X}^{o}$, is the partition $\{Z_{q}\mid q\in\widehat{P}\}^{*}$,
where $Z_{q}=\{w\in X_{q}\mid w\text{ is a clean point\ensuremath{\}}}$
for $q\in\widehat{P}$. By~\cite[Proposition~4.10]{Apps-Stone},
this is a trim $\widehat{P}$-partition of $W$ and gives rise to
the same trim sets as $\mathscr{X}$.
\end{definition}

\begin{remark}
\label{rem:limit points}If $\{X_{p}\mid p\in P\}^{*}$ is a semi-trim
$P$-partition of the $\omega$-Stone space $W$ then:
\begin{description}
\item [{TP7\label{TP7}}] If $w\in X_{p}$ and $V\subseteq V_{w}$ is a
neighbourhood base of $w$ then $p=\sup_{P}I_{w}=\sup_{P}\{t(A)\mid A\in V\}$
(\cite[Proposition~2.13]{Apps-Stone});
\item [{TP8\label{TP8}}] If $\widehat{P}$ is countable, elements of $P-\widehat{P}$
are never separated in $P$, as if $w\in X_{p}$ is a limit point
then $p=\sup_{P}I_{w}$ and $I_{w}$ is not a principal ideal;
\item [{TP9\label{TP9}}] $W$ is primitive (from Theorem~\ref{Primitive=00003Dtrim},
as the clean interior of $\{X_{p}\}^{*}$ is a trim $\widehat{P}$-partition
of~$W$).
\end{description}
\end{remark}

The following example illustrates the differences between these three
types of partitions.
\begin{example}
\label{exa:Partition types}Let $P=\{r,p_{1},p_{2},\ldots,q_{1},q_{2},\ldots\}$,
where $\{p_{n}\}$ and $\{q_{n}\}$ are increasing sequences with
$r=\sup\{p_{n}\}$. Let $Q_{1}=P\cup\{s,t\}$, where $p_{n}<s<r$
and $q_{n}<t$ for all $n$; let $Q_{2}=P\cup\{s\}$, where $p_{n}<s<r$
and $q_{n}<s$ for all $n$; and let $Q_{3}=P\cup\{t\}$, where $q_{n}<t$
for all $n$. Each $Q_{j}$ inherits the order relations of $P$,
and we set $P^{d}=Q_{j}^{d}=\emptyset$.

By~\cite[Theorem~5.8]{Apps-Stone} there is an $\omega$-Stone space
$W$ with a trim $P$-partition $\mathscr{X}=\{X_{p}\mid p\in P\}^{*}$;
let $U$ and $V$ denote the limit points corresponding to the ideals
$\{p_{n}\}$ and $\{q_{n}\}$ of $P$ respectively.

Define complete $Q_{j}$-partitions $\mathscr{Y}_{j}=\{Y_{ju}\mid u\in Q_{j}\}$
of $W$ for $j\leqslant3$ as follows. Let $Y_{ju}=X_{u}$ for $u\in P$,
except that $Y_{3r}=X_{r}\cup U$; let $Y_{1s}=U$, $Y_{1t}=V$, $Y_{2s}=U\cup V$
and $Y_{3t}=V$. The clean interior of each $\mathscr{Y}_{j}$ is
just~$\mathscr{X}$, with $\widehat{Q_{j}}=P$.

$\mathscr{Y}_{1}$ and $\mathscr{Y}_{2}$ are strongly semi-trim as
all points in $Y_{ju}$ are clean for $u\in P$ and $j=1,2$. Of these,
$\mathscr{Y}_{1}$ is the more natural completion of $\mathscr{X}$;
indeed, $Q_{1}$ is the ``ideal completion'' of $P$ and $\mathscr{Y}_{1}$
is the ``ideal completion'' of $\mathscr{X}$ (see Section~\ref{subsec:The-ideal-completion}
below), whereas $\mathscr{Y}_{2}$ artificially combines $U$ and
$V$ into a single partition element. $\mathscr{Y}_{3}$ is semi-trim
but not strongly semi-trim, as $Y_{3r}$ contains both clean and limit
points; $Q_{3}$ is in fact the ``chain completion'' of $P$ and
$\mathscr{Y}_{3}$ is the ``chain completion'' of $\mathscr{X}$,
as defined in~\cite[Section~4.2]{Apps-Stone}.
\end{example}

\begin{definition}[\emph{Rank diagrams}]
Let $W$ be a compact $\omega$-Stone space, and $(P,\prec)$ a PO
system with a largest element. We recall from~Myers~\cite[Definition~9.1]{Myers}
that $P$ is a \emph{rank diagram} for $W$, and say that \emph{$P$
ranks $W$ via $f$}, if there is a surjective function $f\colon W\rightarrow P$
such that for all $p,q\in P$ and $x\in W$,
\begin{description}
\item [{(R1)\label{R1}}] if $p\succ q$ and $f(x)=p$, then we can find
$\{x_{k}\mid k\geqslant1\}$ such that $x_{k}\neq x$ (all $k$),
$f(x_{k})=q$ and $x=\lim x_{k}$;
\item [{(R2)\label{R2}}] if $x=\lim x_{k}$, $x_{k}\neq x$ (all $k$),
and $f(x_{k})=q$ then $f(x)\succ q$;
\item [{(R3)\label{R3}}] if $q$ is separated, $x=\lim x_{k}$, $f(x)\preccurlyeq q$
and $\{f(x_{k})\mid k\geqslant1\}$ are either pairwise strictly incomparable
or strictly decreasing, then $f(x_{k})\preccurlyeq q$ for almost
all $k$;
\item [{(R4)\label{R4}}] if $f(x)$ is not separated, then we can write
$x=\lim x_{k}$, with $f(x_{k})\succneqq f(x)$, all $k$;
\item [{(R5)\label{R5}}] if $P_{\max}=\{q\}$ and $q\nprec q$, then there
is a unique $w\in W$ such that $f(w)=q$.
\end{description}
\begin{definition}[Orbit diagram]
Let $W$ be an $\omega$-Stone space (not necessarily compact). For
$w\in W$, let $[w]$ denote the orbit of $w$ under homeomorphisms
of $W$, and let $\mathscr{O}(W)=\{[w]\mid w\in W\}$, \emph{the orbit
diagram }of $W$ (see~\cite[section 15]{Myers}), which becomes a
PO system with relation $[w]<[x]$ iff some point in $[w]$ is a limit
of distinct points in $[x]$. 
\end{definition}

\end{definition}

\begin{theorem}[Rank diagrams]
 (Myers~\cite[Theorems~15.3~and~9.7]{Myers})
\begin{enumerate}
\item A compact $\omega$-Stone space has a rank diagram iff it is primitive
and PI\@;
\item a compact PI primitive space $W$ is ranked by $\widetilde{\mathscr{O}(W)}$
(via the map sending an element to its orbit);
\item compact $\omega$-Stone spaces with a common rank diagram are homeomorphic.
\end{enumerate}
\end{theorem}

We will adapt Myers' definition by dropping requirement R5, and the
needs for $P$ to have a largest element and for the space to be compact,
and (for the reasons in Remark~\ref{rem: order reversal}) reversing
the order relation. Hence the following definition means that $\mathscr{X}$
is a rank partition of \emph{$W$} iff the label function $\tau$
satisfies properties R1 to R4 above, after applying order reversal,
as conditions R1 and R2 are equivalent to $\{X_{p}\mid p\in P\}$
being a $\widetilde{P}$-partition (i.e.\ $p\succ q$ iff $X_{p}\subseteq X_{q}^{\prime}$),
where $X_{p}=\{x\in W\mid f(x)=p\}$.
\begin{definition}[\emph{Rank partitions}]

Let $W$ be an $\omega$-Stone space, $P$ a PO system and $\mathscr{X}$
a complete $P$-partition of $W$ with label function $\tau$. Then
$\mathscr{X}$ is a \emph{rank $P$-partition of $W$} if also:
\begin{description}
\item [{(RP1)\label{RP1}}] if $q$ is separated, $x=\lim x_{k}$, $\tau(x)\geqslant q$
and $\{\tau(x_{k})\mid k\geqslant1\}$ are either pairwise strictly
incomparable or strictly increasing, then $\tau(x_{k})\geqslant q$
for almost all $k$;
\item [{(RP2)\label{RP2}}] if $\tau(x)$ is not separated, then we can
write $x=\lim x_{k}$, with $\tau(x_{k})\lneqq\tau(x)$ for all $k$.
\end{description}
\end{definition}

\begin{example}
For the partitions considered in Example~\ref{exa:Partition types},
$\mathscr{Y}_{2}$ fails condition~RP1, for we can take $q=p_{1}$,
$x\in V$ and $x_{k}\in X_{q_{k}}$, as $\tau(x)=s\geqslant p_{1}$,
but $\tau(x_{k})=q_{k}\ngeqslant p_{1}$. Partition $\mathscr{Y}_{3}$
fails condition~RP2, taking $x\in X_{r}$, as $r$ is not separated
in $Q_{3}$, but if $x=\lim x_{k}$ then $\tau(x_{k})=\tau(x)=r$
for sufficiently large $k$. Partition $\mathscr{Y}_{1}$ however
is a rank $Q_{1}$-partition of~$W$.
\end{example}

Our first result in this section shows how rank partitions relate
to strongly semi-trim partitions. 
\begin{theorem}
\label{Propn rank diagram}

Let $W$ be an $\omega$-Stone space, $P$ a PO system and $\mathscr{X}$
a complete $P$-partition of $W$. Then $\mathscr{X}$ is a rank $P$-partition
of $W$ iff $\mathscr{X}$ is strongly semi-trim and all elements
of $\widehat{P}$ are compact in $P$.
\end{theorem}

\begin{proof}
(i) Suppose $\mathscr{X}=\{X_{p}\mid p\in P\}$ is a rank $P$-partition
of $W$, and let $R=\Co(W)$. We need to show that properties T1,
T2 and SST3 hold and that all elements of $\widehat{P}$ are compact
in $P$. 

Suppose $A\in R$ and $y\in X_{p}\cap A$. We claim that we can find
$q\in P$ and a $q$-trim $C\in R$ such that $y\in C\subseteq A$,
with $q=p$ if $p$ is separated, so that T1 holds.

If $p$ is not separated, apply~\cite[Lemma~9.2(a),(b)]{Myers} in
turn within the compact set~$A$ (noting that the proofs of~\cite[Lemma~9.2(a),(b)]{Myers}
do not use property~R5) to find $q\in T(A)_{\min}$ and $x\in A\cap X_{q}$
such that $p\gneqq q$, where $q$ is separated. If $p$ is separated,
let $q=p$ and $x=y$.

Use~\cite[Lemma~9.2(c)]{Myers} (taking $x^{\prime}=y$ if $y\neq x)$
to find $B\in R$ such that $x,y\in B$ with $\tau(z)>q$ for all
$z\in B$ with $z\neq x$. Hence $T(A\cap B)\subseteq\{r\in P\mid r\geqslant q\}$.
But $T(A\cap B)\supseteq\{r\in P\mid r\geqslant q\}$ as $x\in A\cap B$
and $\mathscr{X}$ is a $P$-partition. Also, if $q\nless q$ then
$B\cap X_{q}=\{x\}$. Hence $C=A\cap B$ is the required $q$-trim
neighbourhood of $y$. 

Moreover, if $p$ is separated then $p\in\widehat{P}$ and $y$ is
a clean point. Conversely if $A$ is $p$-trim, $x\in A\cap X_{p}$,
and $x=\lim x_{k}$ then $x_{k}\in A$ and $\tau(x_{k})\geqslant p$
for large enough $k$, and so $p$ is separated by RP2. Hence $p$
is separated iff $p\in\widehat{P}$, and SST3 follows as all points
in $X_{p}$ are clean. If $x\in X_{q}$ has a neighbourhood base of
$p$-trim sets for some $p\in\widehat{P}$, then $q\geqslant p$ and
$x\in\overline{X_{p}}$, so $q=p$; hence $\mathscr{X}$ is full,
T2 holds and $\mathscr{X}$ is strongly semi-trim. Moreover, $\widehat{P}$
is countable (TP4).

Suppose next that $p\in\widehat{P}$ and $J$ is an ideal of $\widehat{P}$
such that $p\leqslant\sup_{P}J$. If $J$ is principal, then clearly
$p\in J$. Otherwise find $w\in W$ such that $I_{w}=J$ and $\tau(w)=\sup_{P}J$
(TP3 and TP7), and let $A_{1}\supseteq A_{2}\supseteq\cdots$ be
a neighbourhood base of $w$ of trim sets such that $\{t(A_{k})\mid k\geqslant1\}$
is strictly increasing (choosing a sub-sequence of $\{A_{k}\}$ if
necessary). Choose $x_{k}\in A_{k}\cap X_{t(A_{k})}$ so that $w=\lim x_{k}$
and $\tau(x_{k})=t(A_{k})\in J$. Now $p$ is separated, so by RP1
$p\leqslant\tau(x_{k})$ for almost all $k$, and so $p\in J$.

Suppose finally that $p\in\widehat{P}$ and $J$ is an ideal of $P$
such that $p\leqslant\sup_{P}J$. Let $K=\bigcup\{L\mid L\text{ is an ideal of }\widehat{P}\text{ and }\sup_{P}L\in J\}$.
Clearly $K\subseteq J$ and $K$ is a lower subset of $\widehat{P}$.
If $L_{1}$ and $L_{2}$ are ideals of $\widehat{P}$ such that $\sup_{P}L_{i}\in J$
($i=1,2$), choose $q\in J$ such that $\sup_{P}L_{i}\leqslant q$
($i=1,2$) and find $w\in X_{q}$ so that $q=\sup_{P}I_{w}$ (TP7).
For each $r\in L_{i}$, $r\leqslant q=\sup_{P}I_{w}$, so that $r\in I_{w}$
by the previous paragraph. Hence $L_{i}\subseteq I_{w}\subseteq K$
($i=1,2$) and $K$ is an ideal of $\widehat{P}$. Moreover, for each
$q\in J$ we have $q=\sup_{P}I_{w}$ for $w\in X_{q}$, with $I_{w}\subseteq K$,
and so $\sup_{P}J=\sup_{P}K$. Hence $p\in K\subseteq J$, and $p$
is compact in $P$. 

(ii) Now suppose that $\mathscr{X}$ is a complete strongly semi-trim
$P$-partition of $W$ and that all elements of $\widehat{P}$ are
compact (and hence separated) in $P$. $\widehat{P}$ is countable
by~TP4. We must show that RP1 and RP2 hold. 

For RP1, suppose $q$ is separated, $x=\lim x_{k}$, and $\tau(x)=\sup_{P}I_{x}\geqslant q$.
By~TP8, $q\in\widehat{P}$, so $q$ is compact and $q\in I_{x}$.
Choose $A\in V_{x}$ such that $t(A)=q$; then $\tau(x_{k})\geqslant q$
for all sufficiently large $k$ (such that $x_{k}\in A$).

For RP2, let $A_{1}\supseteq A_{2}\supseteq\cdots$ be a neighbourhood
base of $x$ of trim sets, so that $\tau(x)=\sup\{t(A_{k})\mid k\geqslant1\}$,
with $\{t(A_{k})\}$ increasing. Choose $x_{k}\in A_{k}$ such that
$\tau(x_{k})=t(A_{k})\in\widehat{P}$; we have $x=\lim x_{k}$. If
$\tau(x)$ is not separated in $P$, then $\tau(x)\notin\widehat{P}$
(as elements of $\widehat{P}$ are compact and so separated in $P$),
and hence $\tau(x_{k})\lneqq\tau(x)$ for all $k$, as required.
\end{proof}
\begin{remark}
These ideas lead to an alternative definition of trim partitions.
Let $W$ be an $\omega$-Stone space, $P$ a PO system, and $\mathscr{X}=\{X_{p}\mid p\in P\}^{*}$
a $P$-partition of $W$ with label function $\tau\colon X_{P}\rightarrow P$.
Similar arguments to those in the previous proof show that $\mathscr{X}$
is a trim $P$-partition of $W$ if and only if it satisfies the following
two conditions:
\begin{enumerate}
\item if $x\in X_{P}$ and $x=\lim x_{k}$, where $x_{k}\in X_{P}-\{x\}$,
then $\tau(x_{k})>\tau(x)$ for almost all $k$;
\item if $x\notin X_{P}$, we can write $x=\lim x_{k}$, where $x_{k}\in X_{P}$
and $\{\tau(x_{k})\}$ are strictly increasing, and such that if $x=\lim y_{n}$
where $y_{n}\in X_{P}$, then for each $k$ we have $\tau(x_{k})\leqslant\tau(y_{n})$
for almost all $n$.
\end{enumerate}
\end{remark}

\begin{proposition}
\label{Prop: completeness condns}Let $P$ be a PO system and $\mathscr{X}$
a complete semi-trim $P$-partition of the $\omega$-Stone space $W$.
\begin{enumerate}
\item \label{enu:complete1}$\mathscr{X}$ is a strongly semi-trim partition
iff $\widehat{P}$ is weakly separated in $P$;
\item \label{enu:complete2}$\mathscr{X}$ is a rank partition of $W$ iff
every element of $\widehat{P}$ is compact in $P$.
\end{enumerate}
\end{proposition}

\begin{proof}
\ref{enu:complete1} If $\mathscr{X}$ is strongly semi-trim, $p\in\widehat{P}$
and $p=\sup_{P}J$ for some ideal $J$ of $\widehat{P}$, then we
can find $w\in W$ such that $I_{w}=J$ (TP3), so that $w\in X_{p}$.
Hence $w$ is a clean point and $p\in J$ (TP2). Conversely, if $\widehat{P}$
is weakly separated in $P$ and $w\in X_{p}$ for $p\in\widehat{P}$,
then $p=\sup_{P}I_{w}$, so $p\in I_{w}$ and $w$ is a clean point. 

\ref{enu:complete2} \textquotedbl Only if\textquotedbl{} is immediate
from Theorem~\ref{Propn rank diagram}. \textquotedbl If\textquotedbl{}
follows from~\ref{enu:complete1} and Theorem~\ref{Propn rank diagram},
as the conditions mean that $\widehat{P}$ is weakly separated in~$P$:
for if $p\in\widehat{P}$ and $p=\sup_{P}J$ for an ideal $J$ of
$\widehat{P}$, then $p=\sup_{P}K$, where $K=\{q\in P\mid q\leqslant r\text{, some }r\in J\}$;
but $K$ is an ideal of $P$, so $p\in K$ by compactness and $p\in J$.
\end{proof}
\begin{remark}
The previous two results effectively establish a hierarchy for complete
partitions: whether a semi-trim $P$-partition is strongly semi-trim
or a rank partition depends only on how $\widehat{P}$ is placed within
$P$.
\end{remark}

\subsection{\label{subsec:Regular-extensions}Regular extensions of a semi-trim
partition}

In what follows, we will be considering complete partitions that extend
a trim partition, for which the following definition will be useful:
\begin{definition}
\label{def: regular extn}If $P$ and $Q$ are PO systems with $Q\subseteq P$,
and $\mathscr{X}=\{X_{p}\mid p\in P\}^{*}$ and $\mathscr{Y}=\{Y_{q}\mid q\in Q\}^{*}$
are a $P$-partition and semi-trim $Q$-partition respectively of
the $\omega$-Stone space~$W$, we say that $\mathscr{X}$ is a \emph{regular
extension }of $\mathscr{Y}$ (per~\cite[Definition~4.4]{Apps-Stone})
if:
\begin{enumerate}
\item $Y_{q}\subseteq X_{q}$ for each $q\in Q$; 
\item for all $A\in\Co(W)$ and $p\in T_{\mathscr{X}}(A)$, we can find
$q\in T_{\mathscr{Y}}(A)$ such that $q\leqslant p$.
\end{enumerate}
\end{definition}

The above regularity definition is the same as that already encountered:
\begin{lemma}[Regular extensions]
\label{Lemma:regext}

Let $P$ and $Q$ be PO systems with $Q\subseteq P$, and let $\mathscr{X}=\{X_{p}\mid p\in P\}^{*}$
and $\mathscr{Y}=\{Y_{q}\mid q\in Q\}^{*}$ be a $P$-partition and
semi-trim $Q$-partition respectively of the Stone space~$W$ such
that $\mathscr{Y}$ refines $\mathscr{X}$ via $\alpha$, where $\alpha\colon Q\rightarrow P$
is the inclusion map. Then the following are equivalent:
\begin{enumerate}
\item \label{enu:regext1}$\mathscr{X}$ is a regular extension of $\mathscr{Y}$;
\item \label{enu:regext4}$\mathscr{X}$ is a semi-trim $P$-partition,
$\widehat{P}=\widehat{Q}$, $P^{d}=Q^{d}$ and $\Trim_{p}(\mathscr{Y})=\Trim_{p}(\mathscr{X})$
for all $p\in\widehat{P}$;
\item \label{enu:regext5}$\mathscr{X}$ is a semi-trim $P$-partition and
$\mathscr{Y}^{o}=\mathscr{X}^{o}$;
\item \label{enu:regext2}$\mathscr{Y}\prec_{\alpha}\mathscr{X}$.
\end{enumerate}
\end{lemma}

\begin{proof}
\ref{enu:regext1}~$\Rightarrow$~\ref{enu:regext4} Immediate from~\cite[Lemma~4.9]{Apps-Stone}.

\ref{enu:regext4}~$\Rightarrow$~\ref{enu:regext5} $\mathscr{X}$
and $\mathscr{Y}$ have the same trim sets. Also, if $p\in\widehat{P}$
and $x\in X_{p}$ is clean in $\mathscr{X}$, then by fullness $x\in Y_{p}$
and so $x$ is clean in $\mathscr{Y}$. Therefore $\mathscr{X}$ and
$\mathscr{Y}$ have the same clean points and $\mathscr{Y}^{o}=\mathscr{X}^{o}$.

\ref{enu:regext5}~$\Rightarrow$~\ref{enu:regext2} Suppose $p\in\widehat{P}$
and $A\in\Trim_{p}(\mathscr{X})$. Choose $w\in A\cap X_{p}$. Then
$w$ is clean in $\mathscr{X}$ and $\mathscr{Y}$, and so $w\in A\cap Y_{p}$
and $p\in T_{\mathscr{Y}}(A)$.

\ref{enu:regext2}~$\Rightarrow$~\ref{enu:regext1} Immediate from
the generic regularity definition: namely that $\mathscr{Y}\prec_{\alpha}\mathscr{X}$
provided $T_{\mathscr{X}}(A)\subseteq(T_{\mathscr{Y}}(A)\alpha)_{\uparrow}$
for all compact open $A$.
\end{proof}

\subsection{\label{subsec:The-ideal-completion}The ideal completion of a trim
partition}

Our next definition will lead to an alternative characterisation of
rank partitions.
\begin{definition}[\emph{Ideal completions}]

If $P$ is a poset, the \emph{ideal completion of $P$ }is the poset
$\id(P)$ whose elements are the ideals of $P$, ordered by inclusion.
We identify $P$ with the subset of $\id(P)$ consisting of the principal
ideals of $P$. If $P$ is a PO system, then $\id(P)$ becomes a PO
system with the definition $\id(P)^{d}=\{\{p\}_{\downarrow}\mid p\in P^{d}\}$.

If $P$ is a countable PO system and $\mathscr{X}$ is a trim $P$-partition
of the $\omega$-Stone space~$W$, the \emph{ideal completion of
$\mathscr{X}$ }is the complete partition $\id(\mathscr{X})=\{Y(J)\mid J\in\id(P)\}$
of~$W$, where $Y(J)=\{w\in W\mid I_{w}=J\}$ for an ideal $J$ of
$P$, noting that each $Y(J)$ is non-empty (TP3).
\end{definition}

\begin{remark}
If $\mathscr{X}=\{X_{p}\mid p\in P\}^{*}$, then (with the above notation)
we have $Y(p_{\downarrow})=X_{p}$, as $\mathscr{X}$ is trim. We
use different indexing for the elements of the trim partition and
its ideal completion, so that notationally $X_{J}=\bigcup\{X_{p}\mid p\in J\}$,
whereas $Y(J)=\{w\in W\mid I_{w}=J\}$.

In~\cite[section~4.2]{Apps-Stone} we defined the \emph{chain completion}
of a semi-trim $P$-partition, whose underlying PO system is the chain
completion of $P$. Unlike the ideal completion, the chain completion
of $P$ identifies an ideal of $P$ with its supremum in~$P$, where
this exists. The chain completion can be defined for any semi-trim
partition, whereas the ideal completion is naturally defined only
for trim partitions. However, as we shall see shortly, the ideal completion
provides an explicit way of constructing rank partitions.
\end{remark}

\begin{lemma}
\label{Lemma:compact}If $P$ is a PO system, then elements of $P$
are compact in $\id(P)$.
\end{lemma}

\begin{proof}
Suppose $p\in P$ and $\{p\}_{\downarrow}\subseteq\sup_{\id(P)}K$,
where $K$ is an ideal of $\id(P)$. We claim that $\bigcup_{J\in K}J$
is an ideal of $P$. For suppose $q_{1}\in J_{1}$ and $q_{2}\in J_{2}$,
where $J_{i}\in K$ for $i=1,2$. As $K$ is an ideal of $\id(P)$,
we can find $J_{3}\in K$ such that $J_{i}\subseteq J_{3}$ for $i=1,2$;
and we can then find $q_{3}\in J_{3}$ such that $q_{i}\leqslant q_{3}$
for $i=1,2$. It follows that $\sup_{\id(P)}K=\bigcup_{J\in K}J$.
Hence $p\in J$ for some $J\in K$, $\{p\}_{\downarrow}\subseteq J$
and so $\{p\}_{\downarrow}\in K$, as required.
\end{proof}
We are ready to describe the properties of the ideal completion of
a trim partition.
\begin{theorem}[Ideal completions]
\label{Theorem: ideal completion properties}Let $W$ be a primitive
$\omega$-Stone space, $P$ a PO system and $\mathscr{Y}$ a trim
$P$-partition of $W$. Then $\id(\mathscr{Y})$ is a rank $\id(P)$-partition
of $W$. Moreover:
\begin{enumerate}
\item \label{enu:SST1}$(\id(\mathscr{Y}))^{o}=\mathscr{Y}$;
\item \label{enu:SST2}$\id(\mathscr{Y})$ regularly extends $\mathscr{Y}$;
\item \label{enu:SST3}the elements of $\id(\mathscr{Y})$ are the orbits
of $P$-homeomorphisms of $W$;
\item \label{enu:SST4}if $\mathscr{V}$ is a complete semi-trim $Q$-partition
of $W$ such that $\mathscr{V}^{o}=\mathscr{Y}$ (so that $\widehat{Q}=P$),
then $\id(\mathscr{Y}^{o})\prec_{\beta}\mathscr{V}$, where $J\beta=\sup_{Q}J$
for $J\in\id(P)$;
\item \label{enu:SST5}if $Q$ is a PO system, $\mathscr{V}$ a trim $Q$-partition
of $W$ and $\alpha\colon Q\rightarrow P$ a surjective morphism such
that $\mathscr{V}\prec_{\alpha}\mathscr{Y}$, then there is a map
$\id(\alpha)\colon\id(Q)\rightarrow\id(P)$ such that $\id(\mathscr{V})\prec_{\id(\alpha)}\id(\mathscr{Y})$.
\end{enumerate}
\end{theorem}

\begin{proof}
Let $\mathscr{Y}=\{Y_{p}\mid p\in P\}^{*}$ and $\mathscr{Z}=\id(\mathscr{Y})=\{Z(J)\mid J\in\id(P)\}$.
We show first that $\mathscr{Z}$ is an $\id(P)$-partition of $W$.
If $J,K\in\id(P)$ with $J<K$, and $w\in Z(J)$, choose any $A\in\Trim(\mathscr{Y})$
such that $w\in A$. Then $t(A)\in J\subseteq K$, so $A\cap Z(K)\neq\emptyset$
(TP3). Moreover if $J=K$ then $(A-\{w\})\cap Z(J)\neq\emptyset$:
for if $J=p_{\uparrow}$ with $p<p$ then $Y_{p}$ has no isolated
points, while if $J$ is non-principal we can find $B\in\Trim_{q}(\mathscr{Y})$
such that $w\in B\subseteq A$ for some $q\gneqq t(A)$ and now $(A-B)\cap Z(J)\neq\emptyset$
as $A-B$ is still $t(A)$-trim. Hence $(A-\{w\})\cap Z(K)\neq\emptyset$
and so $Z(J)\subseteq Z(K)^{\prime}$. Conversely, if $J\nless K$,
find $w\in Z(J)$, and $p\in I_{w}=J$ such that $p\notin K$ if $J\neq K$;
if $J=K=q_{\downarrow}$ with $q\in P^{d}$, let $p=q$. Let $A$
be a $p$-trim neighbourhood of $w$. Then $(A-\{w\})\cap Z(K)=\emptyset$,
so $w\notin Z(K)^{\prime}$. Hence $\mathscr{Z}$ is an $\id(P)$-partition
of $W$. 

Moreover, if $A$ is compact open and $w\in A\cap Y_{p}$ ($p\in P$),
then $I_{w}=p_{\downarrow}$, $w\in Z(I_{w})$, and $I_{w}$ is just
the image of $p$ in $\id(P)$. Hence $\mathscr{Z}$ is a regular
extension of $\mathscr{Y}$. By Lemma~\ref{Lemma:regext} $\mathscr{Z}$
is a semi-trim $\id(P)$-partition of $W$, and gives rise to the
same trim sets as $\mathscr{Y}$, so $\mathscr{Z}^{o}=\mathscr{Y}$.
Hence by Proposition~\ref{Prop: completeness condns} and Lemma~\ref{Lemma:compact},
$\mathscr{Z}$ is a rank $\id(P)$-partition of $W$.

We have established~\ref{enu:SST1} and~\ref{enu:SST2}, and~\ref{enu:SST3}
follows immediately from~\cite[Corollary~3.6(ii)]{Apps-Stone}.

For~\ref{enu:SST4}, let $Q$ be a PO system containing $P$ and
$\mathscr{V}=\{V_{q}\mid q\in Q\}$ a complete semi-trim $Q$-partition
of $W$ such that $\mathscr{V}^{o}=\mathscr{Y}$. If $x,y\in W$ and
$I_{x}=I_{y}$, then $x,y\in V_{q}$ where $q=\sup_{Q}I_{x}$ (TP7).
Hence $\id(\mathscr{Y})$ refines $\mathscr{V}$ via $\beta$, and
the refinement is regular by Lemma~\ref{Lemma:regext} as $\mathscr{V}^{o}=\mathscr{Y}=(\id(\mathscr{Y}))^{o}$.

For~\ref{enu:SST5}, let $\id(\alpha)\colon\id(Q)\rightarrow\id(P)$
be given by $J\id(\alpha)=(J\alpha)_{\downarrow}$ for $J\in\id(Q)$;
an easy check shows that $\id(\alpha)$ extends $\alpha$ and $(J\alpha)_{\downarrow}$
is an ideal of $P$. Choose $w\in W$ and let $I_{w}(\mathscr{V})=\{q\in Q\mid w\in A\text{ and }A\in\Trim_{q}(\mathscr{V})\text{ for some }A\}$,
and similarly for $I_{w}(\mathscr{Y})$. We claim that $I_{w}(\mathscr{V})\id(\alpha)=I_{w}(\mathscr{Y})$.
For if $w\in A$ and $A\in\Trim_{q}(\mathscr{V})$, then $A\in\Trim_{q\alpha}(\mathscr{Y})$;
hence $I_{w}(\mathscr{V})\id(\alpha)\subseteq I_{w}(\mathscr{Y})$.
Conversely, if $w\in A$ and $A\in\Trim_{p}(\mathscr{Y})$, then we
can find $B\in\Trim_{q}(\mathscr{V})$, say, such that $w\in B\subseteq A$;
and then $B\in\Trim_{q\alpha}(\mathscr{Y})$ and $p\leqslant q\alpha\in I_{w}(\mathscr{V})\id(\alpha)$.
Therefore $\id(\mathscr{V})$ refines $\id(\mathscr{Y})$ via $\id(\alpha)$,
and the refinement is regular as the map $\alpha$ between the underlying
trim partitions is regular.
\end{proof}
The following Corollary makes precise the link between rank partitions
and trim partitions.
\begin{corollary}[Ideal completions]
\label{Theorem: rank partition =00003D ideal completion}Let $W$
be a primitive $\omega$-Stone space. Then the maps $\mathscr{Y}\mapsto\id(\mathscr{Y})$
and $\mathscr{Z}\mapsto\mathscr{Z}^{o}$ are inverse bijections between
the set of trim partitions $\mathscr{Y}$ and the set of rank partitions
$\mathscr{Z}$ of $W$. In particular, the rank partitions of $W$
take precisely the form $\id(\mathscr{Y})$, where $\mathscr{Y}$
is a trim partition of $W$.
\end{corollary}

\begin{proof}
By Theorem~\ref{Theorem: ideal completion properties}, if $\mathscr{Y}$
is a trim partition of $W$ then $\id(\mathscr{Y})$ is a rank partition
of $W$ and $(\id(\mathscr{Y}))^{o}=\mathscr{Y}$. Moreover, if $\mathscr{Z}$
is a rank $Q$-partition, say, of $W$ then $\mathscr{Z}^{o}$ is
a trim $\widehat{Q}$-partition and $\id(\mathscr{Z}^{o})\prec_{\beta}\mathscr{Z}$,
where $J\beta=\sup_{Q}J$ for $J\in\id(\widehat{Q})$. But elements
of $\widehat{Q}$ are compact in $Q$, by Proposition~\ref{Prop: completeness condns}.
So if $J$ and $K$ are ideals of $\widehat{Q}$ such that $\sup_{Q}J=\sup_{Q}K$,
then $J=K$. So $\beta$ is a bijection, and $\id(\mathscr{Z}^{o})=\mathscr{Z}$
as the two partitions are complete, as required.
\end{proof}
\begin{corollary}[Orbit diagrams]
\label{Thm: orbit diagram}

Let $W$ be a primitive $\omega$-Stone space. Then $\id(\mathscr{X}(W))$
is the orbit diagram $\mathscr{O}(W)$ of~$W$. Moreover, $\mathscr{O}(W)$
regularly consolidates $\id(\mathscr{Y})$ for every trim partition
$\mathscr{Y}$ of $W$.
\end{corollary}

\begin{proof}
Let $P=\mathscr{S}(W)$ and $\id(\mathscr{X}(W))=\{X(J)\mid J\in\id(P)\}$.
The first assertion follows from Theorem~\ref{Theorem: ideal completion properties}\ref{enu:SST3}:
homeomorphisms of $W$ preserve the isomorphism types of the compact
opens, and so are also $P$-homeomorphisms; and the order relation
on $\mathscr{O}(W)$ matches that on $\id(\mathscr{X}(W))$, since
$\id(\mathscr{X}(W))$ is an $\id(P)$-partition of $W$ and so for
$J,K\in\id(P)$, $X(J)\subseteq X(K)^{\prime}$ iff $J<K$.

Suppose now that $\mathscr{Y}$ is a trim $Q$-partition of $W$.
By Corollary~\ref{Cor:morphism to canonical}, there is a surjective
morphism $\beta\colon Q\rightarrow\mathscr{S}(W)$ such that $\mathscr{Y}$
regularly refines $\mathscr{X}(W)$ via $\beta$. So by Theorem~\ref{Theorem: ideal completion properties}\ref{enu:SST5}
and the above, $\id(\mathscr{Y})$ regularly refines $\mathscr{O}(W)$,
as required.
\end{proof}
\begin{remark}
The first statement of Corollary~\ref{Thm: orbit diagram} was established
by Hansoul~\cite[Proposition~10]{Hansoul} for the case of compact
PI primitive spaces $W$, for which $\mathscr{O}(W)$ is the canonical
rank partition. The following corollary extends the results of Hansoul
and Myers to locally compact Boolean spaces. However, two $\omega$-Stone
spaces admitting rank $P$-partitions for some PO system $P$ need
not be homeomorphic: as the spaces $X$ and $Y$ in Example~\ref{exa:Iso P, non homeo space}
are not homeomorphic, but $\mathscr{S}(X)\cong\mathscr{S}(Y)$, so
$\id(\mathscr{S}(X))\cong\id(\mathscr{S}(Y))$ and $X$ and $Y$ both
admit rank $\id(\mathscr{S}(X))$-partitions.
\end{remark}

\begin{corollary}[Rank partitions]
Let $W$ be an $\omega$-Stone space. Then the following are equivalent:
\begin{enumerate}
\item $W$ is primitive;
\item $W$ admits a rank partition;
\item the orbit diagram of $W$ forms a rank partition.
\end{enumerate}
\end{corollary}

\begin{proof}
If $W$ is primitive, then $W$ admits a trim partition by Theorem~\ref{Primitive=00003Dtrim},
and hence a rank partition by Theorem~\ref{Theorem: ideal completion properties}.
Conversely, if $W$ admits a rank partition $\mathscr{Z}$, then $\mathscr{Z}$
is semi-trim by Theorem~\ref{Propn rank diagram}, so $W$ is primitive
by~TP9. The equivalence of the second and third statements now follows
from Corollary~\ref{Thm: orbit diagram}.
\end{proof}

\subsection{Images of $\id(P)$ and complete extensions of a trim $P$-partition}

We can characterise complete semi-trim regular extensions of a given
trim $P$-partition of an $\omega$-Stone space as certain images
of $\id(P)$.
\begin{proposition}
\label{Propn maps}Let $W$ be an $\omega$-Stone space, $P$ a PO
system and $\mathscr{Y}$ a trim $P$-partition of $W$. Then there
is a natural correspondence between (i) complete regular semi-trim
extensions $\mathscr{Z}$ of $\mathscr{Y}$ and (ii) surjective order-preserving
maps $\beta\colon\id(P)\rightarrow Q$ for some PO system $Q$ satisfying
the following properties (viewing $P$ as a subset of $Q$):
\begin{enumerate}
\item \label{enu:map1}$\beta$ restricted to $P$ is an isomorphism onto
$P\beta$;
\item \label{enu:map2}if $J$ is an ideal of $P$, then $J\beta=\sup_{Q}J$;
\item \label{enu:map3}if $p\lneqq q$ for $p\in P$ and $q\in Q$, then
there is an ideal $J$ of $P$ such that $p\in J$ and $J\beta=q$;
\item \label{enu:map4}if $p\in P^{d}$ and $p=\sup_{Q}J$ for an ideal
$J$ of $P$ then $p\in J$.
\end{enumerate}
Under this correspondence, $\tau_{\mathscr{Z}}(w)=I_{w}\beta$ for
$w\in W$ if $\mathscr{Z}$ is a $Q$-partition. 

Moreover, $\mathscr{Z}$ is strongly semi-trim iff $J\beta\in P$
only when $J$ is principal; and $\mathscr{Z}$ is a rank partition
iff $\beta$ is an isomorphism, in which case properties~\ref{enu:map1}~to~\ref{enu:map4}
automatically hold.
\end{proposition}

\begin{proof}
Suppose first that $P\subseteq Q$ and that $\mathscr{Z}$ is a complete
semi-trim $Q$-partition of $W$ that regularly extends $\mathscr{Y}$.
The map $\beta\colon\id(P)\rightarrow Q:J\mapsto\sup_{Q}J$ is surjective,
by~TP3~and~TP7. Properties~\ref{enu:map1}~and~\ref{enu:map2}
are immediate. For~\ref{enu:map3}, choose $B\in\Trim_{p}(\mathscr{Y})$,
choose $w\in B$ such that $\tau_{\mathscr{Z}}(w)=q$ (as by Lemma~\ref{Lemma:regext}
$B\in\Trim_{p}(\mathscr{Z})$) and let $J=I_{w}$ (using TP7). For~\ref{enu:map4}
choose $q\in J$, a $q$-trim set $B$, and $w\in B$ such that $I_{w}=J$
(TP3), so that $w\in X_{p}$ (TP7); but $w$ is a clean point by~\cite[Proposition~2.13,~STP5]{Apps-Stone},
so $I_{w}$ is principal (TP2) and $p\in J$. Lastly, $\beta$ is
order preserving: as if $J\leqslant K$ then $J\beta\leqslant K\beta$;
and if $J\beta=K\beta=p\in P^{d}$, noting that $Q^{d}=P^{d}$ by
Lemma~\ref{Lemma:regext}, then $J=K=p_{\downarrow}$ by~\ref{enu:map4}
and $J\nless K$; so if $J<K$ then $J\beta<K\beta$.

Conversely, if $\beta\colon\id(P)\rightarrow Q$ satisfies properties~\ref{enu:map1}~to~\ref{enu:map4},
then conditions (i) to (iv) of~\cite[Theorem~4.11]{Apps-Stone} are
satisfied as:

(i) for each $q\in Q$, choose $J\in\id(P)$ such that $J\beta=q$
and pick $p\in J$, so that $\{p\}_{\downarrow}\subseteq J$ and $p=\{p\}_{\downarrow}\beta\leqslant J\beta=q$;

(ii) if $p\in P$ and $q\in Q$ with $p\lneqq q$, then by~\ref{enu:map3}~and~\ref{enu:map2}
there is an ideal $J$ of $P$ such that $p\in J$ and $q=\sup_{Q}J$;

(iii) $Q$ is $\omega$-complete over $P$, as an increasing sequence
of elements of $P$ generates an ideal of $P$ whose image in $Q$
is their supremum, by~\ref{enu:map2}; 

(iv) is immediate from~\ref{enu:map4}.

Finally, $Q^{d}=P^{d}$: for by~\ref{enu:map1} if $p\in P$ then
$p<p$ in $P$ iff $p<p$ in $Q$; and if $J$ is a non-principal
ideal of $P$ then $J<J$ in $\id(P)$, so $J\beta<J\beta$ and $J\beta\notin Q^{d}$.

Hence by~\cite[Theorem~4.11]{Apps-Stone}, $\mathscr{Y}$ can be
regularly extended to a complete semi-trim $Q$-partition $\mathscr{Z}$,
say, of $W$, with $\tau_{\mathscr{Z}}(w)=\sup_{Q}I_{w}=I_{w}\beta$
for $w\in W$.

The stated necessary and sufficient condition for $\mathscr{Z}$ to
be strongly semi-trim is equivalent to $P$ being weakly separated
in $Q$, and so follows from Proposition~\ref{Prop: completeness condns}\ref{enu:complete1}. 

If $\mathscr{Z}$ is a rank partition then all elements of $P$ are
compact in $Q$ by Proposition~\ref{Prop: completeness condns}\ref{enu:complete2}.
So if $\sup_{Q}J=\sup_{Q}K$ for ideals $J$ and $K$ of $P$ then
$J=K$ (by considering $J_{\downarrow}$ and $K_{\downarrow}$ in~$Q$)
and so $\beta$ is injective. Finally, if $J\beta<K\beta$, then compactness
again gives $J<K$, as if $J=K=\{p\}_{\downarrow}$ for $p\in P^{d}$
then $J\nless J$, and so $\beta$ is an isomorphism. Conversely,
if $\beta$ is an isomorphism, then all elements of $P$ are compact
in $Q$ by Lemma~\ref{Lemma:compact} and so $\mathscr{Z}$ is a
rank partition by Proposition~\ref{Prop: completeness condns}\ref{enu:complete2};
moreover, $\beta$ preserves suprema, so properties~\ref{enu:map1}~to~\ref{enu:map4}
are immediate, as for~\ref{enu:map4} if $\{p\}_{\downarrow}\beta=p=\sup_{Q}J=J\beta$
then $J=\{p\}_{\downarrow}$.
\end{proof}
\begin{remark}
The previous proposition provides a ``recipe'' for constructing
semi-trim and strongly semi-trim partitions. Property~\ref{enu:map4}
above is redundant if we are considering strongly semi-trim or rank
partitions. It is not obvious whether or not property~\ref{enu:map2}
can be dispensed with.

Property~\ref{enu:map3} says that $\beta$ satisfies the morphism
conditions ``on~$P$''. However, $\beta$ need not be a morphism.
For example, let $\{J_{n}\mid n\geqslant1\}$ be copies of $\mathbb{N}$,
let $p_{n}$ be the minimum element of $J_{n}$, and let $P=\bigcup J_{n}$,
with the added relations $p_{n}<p_{n+1}$, with $P^{d}=\emptyset$.
Let $Q=P\cup\{r,s\}$ with added relations $p_{n}<r<r$, $J_{n}<s<s$
and $r<s$, and define $\beta\colon\id(P)\rightarrow Q$ by $p_{\downarrow}\beta=p$
for $p\in P$, $J_{n}\beta=\{s\}$ for each $n\geqslant1$ and $K\beta=r$,
where $K=\{p_{n}\mid n\geqslant1\}$. Then $\beta$ satisfies the
conditions of Proposition~\ref{Propn maps}, and $K\beta=r<s$, but
there is no ideal $L$ of $P$ such that $K<L$ and $L\beta=s$.
\end{remark}

\subsection{The prime topological Boolean algebra of a primitive space}

We recall that the \emph{prime TBA} (prime topological Boolean algebra)
$\mathscr{U}(W)$ of a space $W$ is the smallest complete sub-TBA
of $2^{W}$ under the usual (infinite) Boolean operations together
with the derived set unitary operation $^{\prime}$. We write $Fi(W)$
for the set of subsets of $W$ invariant under all homeomorphisms
of $W$, which is also a complete sub-TBA of $2^{W}$, and whose elements
are unions of elements of $\mathscr{O}(W)$. 

Hansoul~\cite[Theorem~12]{Hansoul} obtained the following result
for the case of compact pseudo-indecomposable $W$; we extend this
to general primitive $\omega$-Stone spaces.
\begin{theorem}[Hansoul]
\label{Fi=00003Dprime TBA}Let $W$ be a primitive $\omega$-Stone
space. Then the partition formed by the atoms of $\mathscr{U}(W)$
is equal to $\mathscr{O}(W)$, and $\mathscr{U}(W)=Fi(W)$. 
\end{theorem}

\begin{proof}
Let $R=\Co(W)$, and let $\mathscr{Y}=\{Y_{q}\mid q\in Q\}$ be the
partition of $W$ consisting of the atoms of $\mathscr{U}(W)$; $Q$
becomes a PO system with the relation $q<r$ iff $Y_{q}\subseteq Y_{r}^{\prime}$,
and $\mathscr{Y}$ is a complete $Q$-partition of $W$, as $Y_{q}\cap Y_{r}^{\prime}$
is either $Y_{q}$ or $\emptyset$ for each $q,r\in Q$ (we note that
$<$ is antisymmetric, by Hansoul~\cite[Lemma~1.11(1)]{Hansoul2}). 

Now $Fi(W)$ is also a complete sub-TBA of $2^{W}$, so $\mathscr{U}(W)\subseteq Fi(W)$.
By Corollary~\ref{Thm: orbit diagram}, the atoms of $Fi(W)$ are
$\mathscr{O}(W)=\{X_{p}\mid p\in\id(P)\}$, say, where $P=\mathscr{S}(W)$.
Hence $\mathscr{O}(W)$ refines~$\mathscr{Y}$: say $X_{p}\subseteq Y_{p\alpha}$,
where $\alpha\colon\id(P)\rightarrow Q$ is surjective; $\alpha$
is also order-preserving, as if $p<r$ then $X_{p}\subseteq X_{r}^{\prime}\subseteq Y_{r\alpha}^{\prime}$,
so $p\alpha<r\alpha$. For $A\in R$, write $T(A)=\{q\in Q\mid A\cap Y_{q}\neq\emptyset\}$.

We claim first that if $p\alpha=J\alpha=q$, say, where $p\in P$
and $J\in\id(P)$ is non-principal, then we can find $r\in J$ such
that $s\alpha=q$ for $r\leqslant s\in J$. For choose $A\in\Trim_{p}(\mathscr{O}(W))$;
then $T(A)=(p_{\uparrow})\alpha=q{}_{\uparrow}$, as $q_{\uparrow}\subseteq T(A)$,
and $\alpha$ is order-preserving. Let $N_{q}=Q-q_{\uparrow}$; then
$Y_{q}\cap\overline{Y_{N_{q}}}=\emptyset$, so any element of $Y_{q}$
has a neighbourhood of type $q_{\uparrow}$ in $\mathscr{Y}$. Choose
$w\in W$ such that $I_{w}=J$ (TP3), so that $w\in Y_{q}$, and
find $C\in\Trim_{r}(\mathscr{O}(W))$, say, such that $w\in C$ and
$T(C)=q_{\uparrow}$, so that $r\in J$. If now $r\leqslant s\in J$,
then $q\leqslant r\alpha\leqslant s\alpha\leqslant J\alpha=q$, and
so $s\alpha=q$ as required.

We show next that $\alpha$ restricted to $P$ is a morphism. For
if $p,q\in P$ and $p\alpha<q\alpha$, then by considering $T(A)$
for $A\in\Trim_{p}(\mathscr{O}(W))$ we can find $J\in\id(P)$ such
that $p_{\downarrow}<J$ and $J\alpha=q\alpha$. If $J$ is non-principal,
find $s\in J$ such that $s\alpha=q\alpha$, and we may choose $s>p$
as $J$ is an ideal; if however $J=s_{\downarrow}$ with $s\in P$,
then $p<s$ and $s\alpha=q\alpha$.

But $P$ is simple. So $\alpha$ is injective on $P$, and by the
above we cannot have $p\alpha=J\alpha$ with $J$ non-principal. Hence
$X_{p}=Y_{p\alpha}\in\mathscr{U}(W)$ for $p\in P$. But now by Lemma~\ref{Lemma:ideal generation},
$X_{p}\in\mathscr{U}(W)$ for all $p\in\id(P)$. So $\mathscr{Y}=\mathscr{O}(W)$
and $\mathscr{U}(W)=Fi(W)$.
\end{proof}
\begin{lemma}[{Hansoul~\cite[Lemma~11]{Hansoul}}]
\label{Lemma:ideal generation}Let $N$ be a PO system, $J$ an ideal
of $N$ and $\{U_{n}\mid n\in N\}^{*}$ a trim $N$-partition of the
$\omega$-Stone space $W$. Then $\{w\in W\mid I_{w}=J\}=\overline{U_{J}}-\bigcup_{K\in M}\overline{U_{K}}$,
where $M=\{K\in\id(N)\mid K\subsetneqq J\}$. 
\end{lemma}

\begin{proof}
It is enough to show that if $w\in W$ and $K$ is an ideal of $N$,
then $w\in\overline{U_{K}}$ iff $I_{w}\subseteq K$. But if $I_{w}\subseteq K$
then each $A\in V_{w}$ meets $U_{K}$ (TP3), while if $I_{w}\nsubseteq K$
then for $r\in I_{w}-K$ we can find $A\in\Trim_{r}(\{U_{n}\}^{*})$
such that $w\in A$, so that $w\notin\overline{U_{K}}$.
\end{proof}
\begin{remark}
The converse of Theorem~\ref{Fi=00003Dprime TBA} is false in general.
For example, one can construct a compact $\omega$-Stone space $W$
with a distinguished point $w_{0}$ such that $X=W-\{w_{0}\}$ is
primitive but $W$ is not, and for which the atoms of both $\mathscr{U}(W)$
and $Fi(W)$ are equal to those of $\mathscr{U}(X)$ (being the same
as $Fi(X)$) together with $\{w_{0}\}$.
\end{remark}

\section{Touraille's characterisation of primitive spaces}

Finally we provide a topological proof of a result of Touraille~\cite[Theorem~4.1]{Touraille primitive},
extended to locally compact $\omega$-Stone spaces. We recall that
the set of open subsets $\mathscr{H}(W)$ of a Stone space~$W$ becomes
a \emph{complete Heyting{*} algebra }(see~\cite{Touraille Heyting}
for full definition) with the definitions for open subsets $\{I_{i}\},I,J$:
$\sup\{I_{i}\}=\bigcup I_{i}$; $\inf\{I_{i}\}=[\bigcap I_{i}]^{o}$,
where $A^{o}$ denotes the interior of a set $A$; $I\rightarrow J=[(W-I)\cup J]^{o}$;
and with the added unary operation $I^{*}=W-(W-I)^{\prime}$, so that
$I^{*}$ contains $I$ and any isolated points of $W-I$. We recall
further that an element $J$ of a complete lattice is \emph{completely
join irreducible (cji)} if whenever $J\leqslant\sup E$ for some subset
$E$ of the lattice, then $J\leqslant I$ for some $I\in E$. Let
$\Inv(W)$ denote the set of open subsets of $W$ invariant under
homeomorphisms of~$W$, which is a complete Heyting{*} subalgebra
of $\mathscr{H}(W)$.

First, we have a proposition that sheds light on how far we can get
towards a trim partition of a general Stone space.
\begin{proposition}
\label{Heyting general}Let $W$ be the Stone space of the Boolean
ring $R$, and let $\mathscr{V}$ be a complete Heyting{*} subalgebra
of $\mathscr{H}(W)$. Let $P=\mathscr{V}_{\mathrm{cji}}$ denote the
cji elements of $\mathscr{V}$. For $A\in R$, let $I(A)$ denote
the smallest element of $\mathscr{V}$ containing $A$. Then there
is a relation $<$ such that $(P,<)$ becomes a PO system, and a $P$-partition
$\mathscr{Y}=\{Y_{U}\mid U\in P\}^{*}$ of $W$ such that:
\begin{enumerate}
\item \label{enu:cji1}the $U$-supertrim sets, namely $\Trim_{U}(\mathscr{Y})\cap2^{U}$,
consist of all $A\in R$ such that $I(A)=U$, with also $|A\cap Y_{U}|=1$
if $U\in P^{d}$, for $U\in P$;
\item \label{enu:cji2}each point of $Y_{U}$ is clean and has a supertrim
neighbourhood;
\item \label{enu:cji4}$U$ is the union of the $U$-supertrim sets for
$U\in P$;
\item \label{enu:cji3}for $A\in R$, $A$ is a disjoint union of supertrim
sets iff $I(A)=\bigcup\{U\mid U\in Q\}$ for some subset $Q$ of $P$.
\end{enumerate}
\end{proposition}

\begin{proof}
For $U\in P$, let $U^{<}=\bigcup\{V\in P\mid V\subsetneqq U\}$,
so that $U^{<}\neq U$, and let $Y_{U}=U-U^{<}$. For $A\in R$, write
$T(A)=\{U\in P\mid A\cap Y_{U}\neq\emptyset\}$; we note that for
$A\subseteq U\in P$, $I(A)=U$ iff $U\in T(A)$. 

Define a relation $<$ on $P$ by $U<V$ iff $Y_{U}\subseteq Y_{V}^{\prime}$;
this determines a PO system as $<$ is antisymmetric (e.g.\ by Naturman~\cite[Lemma~4.2.7]{Naturman}). 

Firstly, if $U\neq V$ ($U,V\in P)$, then $Y_{U}\cap Y_{V}=\emptyset$:
as if $w\in Y_{U}\cap Y_{V}$, we can find $A\in R$ such that $w\in A\subseteq U\cap V$,
and then $I(A)=U=V$.

Next, $\mathscr{Y}$ is a $P$-partition. For clearly $Y_{V}^{\prime}\cap Y_{P}\supseteq\bigcup_{U<V}Y_{U}$.
Suppose $U\nless V$ with $U\neq V$, and choose $w\in Y_{U}-Y_{V}^{\prime}$.
Find $A\in R$ such that $w\in A\subseteq U$ and $A\cap Y_{V}=\emptyset$.
Let $B=U\cap(W-Y_{V})^{o}$, so that $A\subseteq B\in\mathscr{V}$
(as $(W-Y_{V})^{o}=V\rightarrow V^{<}$). Now $I(A)=U$; hence $B=U$,
$U\cap Y_{V}=\emptyset$ and $Y_{U}\cap Y_{V}^{\prime}=\emptyset$.
Similarly, if $U\nless U$, apply the same argument but taking $B=U\cap(W-Y_{U}^{\prime})$,
noting that $Y_{U}^{\prime}=\overline{Y_{U}}^{\prime}$, so $W-Y_{U}^{\prime}=(W-\overline{Y_{U}})^{*}=\left\{ (W-Y_{U})^{o}\right\} ^{*}=(U\rightarrow U^{<})^{*}$,
to obtain $U\cap Y_{U}^{\prime}=\emptyset$. Thus $Y_{V}^{\prime}\cap Y_{P}=\bigcup_{U<V}Y_{U}$
as required.

For~\ref{enu:cji1}, suppose first that $A\in\Trim_{U}(\mathscr{Y})$
and $A\subseteq U$. Then $I(A)=U$ and we must also have $|A\cap Y_{U}|=1$
if $U\in P^{d}$. Conversely, if these conditions hold, then $A\subseteq U$,
$U\in T(A)$ and $T(A)\supseteq U_{\uparrow}$. If $U\nless V$ and
$U\neq V$, then $U\cap Y_{V}=\emptyset$ by the above and so $V\notin T(A)$;
hence $A\in\Trim_{U}(\mathscr{Y})$ as required. 

For~\ref{enu:cji2}, for any $w\in Y_{U}$ ($U\in P)$ we can choose
$A\in R$ such that $w\in A\subseteq U$, so that $T(A)=U_{\uparrow}$
(as above). If also $U\in P^{d}$, then $Y_{U}\cap Y_{U}^{\prime}=\emptyset$,
so $w\notin Y_{U}^{\prime}$ and we can reduce $A$ as necessary to
assume that $|A\cap Y_{U}|=1$. Hence $w$ is a clean point.

For~\ref{enu:cji4}, we use~\ref{enu:cji1} and~\ref{enu:cji2},
and observe that if $A\in R$ is $U$-supertrim and $B\in R$ with
$B\subseteq U^{<}$, then $A\cup B$ is also $U$-supertrim.

For~\ref{enu:cji3}, if $A=A_{1}\dotplus\ldots\dotplus A_{n}$, where
$A_{i}\in\Trim_{U_{i}}(\mathscr{Y})\cap2^{U_{i}}$ ($U_{i}\in P$),
then it is easy to see that $I(A)=\bigcup_{i\leqslant n}I(A_{i})=\bigcup_{i\leqslant n}U_{i}$. 

Conversely, if $I(A)=\bigcup\{U\mid U\in Q\}$ for some subset $Q$
of $P$, then by compactness we can write $I(A)=\bigcup_{i\leqslant n}U_{i}$,
with each $U_{i}\in P$, and with the $U_{i}$ chosen so as to minimise
$n$. Considering each $U_{i}$ as a union of elements of $R$ and
using compactness again, we can write $A=A_{1}\dotplus\ldots\dotplus A_{n}$,
where $A_{i}\in R$ and $A_{i}\subseteq U_{i}$ for each $i$. So
$U_{i}\subseteq I(A)\subseteq\bigcup_{j\leqslant n}I(A_{j})$, and
we must therefore have $I(A_{i})=U_{i}$ as $U_{i}$ is cji and $n$
is minimal (cf the proof of Touraille~\cite[Proposition~3.2]{Touraille primitive}).
Moreover, if $U_{i}\in P^{d}$, then $U_{i}\cap Y_{U_{i}}^{\prime}=\emptyset$
by the above, so $A_{i}\cap Y_{U_{i}}^{\prime}=\emptyset$ and $A_{i}\cap Y_{U_{i}}$
is finite; therefore we can write $A_{i}$ as a disjoint union of
sets $B_{ij}$ such that |$B_{ij}\cap Y_{U_{i}}|=1$, and now each
$B_{ij}$ will be $U_{i}$-supertrim. Hence we can express $A$ as
a disjoint union of supertrim sets, as required. 
\end{proof}
\begin{remark}
\label{Alternative defn}The relation $<$ on $P$ could alternatively
be defined by setting $U<V$ iff $V\subseteq U$ when $U\neq V$,
with $U<U$ iff $Y_{U}$ is not discrete. For if $U<V$, pick $w\in Y_{U}$
and $A\in R$ such that $w\in A\subseteq U$; then $A\cap Y_{V}\neq\emptyset$,
so we can find $B\subseteq A\cap V$ with $B\cap Y_{V}\neq\emptyset$,
so that $V=I(B)\subseteq I(A)=U$; and if $U<U$ then $Y_{U}\subseteq Y_{U}^{\prime}$,
so $Y_{U}$ is not discrete. Conversely, if $V\subseteq U$ with $U\neq V$,
then $U<V$ as otherwise $U\cap Y_{V}=\emptyset$; and if $Y_{U}$
is not discrete then $U\cap Y_{U}^{\prime}\supseteq Y_{U}\cap Y_{U}^{\prime}\neq\emptyset$
and so $U<U$ (in both cases, using results from the above proof that
$\mathscr{Y}$ is a $P$-partition).
\end{remark}

\begin{theorem}[Touraille]
Let $W$ be an $\omega$-Stone space and $\mathscr{V}$ a complete
Heyting{*} subalgebra of $\Inv(W)$. Let $\mathscr{V}_{\mathrm{cji}}$
denote the cji elements of $\mathscr{V}$. Then $W$ is primitive
iff $\mathscr{V}_{\mathrm{cji}}$ is a basis for $\mathscr{V}$. 
\end{theorem}

\begin{proof}
Suppose first that $P=\mathscr{V}_{\mathrm{cji}}$ is a basis for
$\mathscr{V}$. Then by Proposition~\ref{Heyting general}, there
is a relation $<$ such that $(P,<)$ becomes a PO system, and a $P$-partition
$\mathscr{Y}=\{Y_{U}\mid U\in P\}^{*}$ of $W$ which is a trim partition
(as supertrim sets are trim), subject to checking that $\mathscr{Y}$
is full. 

But if $w\in W$ has a neighbourhood base of $U$-trim sets for some
$U\in P$, then $w\in U$: as if $w\in A$ and $A\in\Trim_{U}(\mathscr{Y})$,
then we can find a $V$-supertrim set $B$, say, for some $V\in P$,
and a $U$-trim set $C$ such that $w\in C\subseteq B\subseteq A$
and $B\subseteq V$; hence $U=V$ and $w\in U$. Moreover, $w\notin U^{<}$
as otherwise we could find a $U$-trim subset of $U^{<}$; hence $w\in Y_{U}$
and $\mathscr{Y}$ is full. So $W$ admits a trim partition and is
primitive.

Suppose now that $W$ is primitive. We will show, as per~\cite[Proposition~2.7]{Touraille primitive},
that if $A\in R$ is PI then $I(A)$ is cji, where $I(A)$ denotes
the smallest element of $\mathscr{V}$ containing~$A$. By compactness,
it is enough to show that if $I(A)\subseteq U_{1}\cup U_{2}$, where
$U_{i}\in\mathscr{V}$, then $I(A)\subseteq U_{i}$ for $i=1$ or
$2$. As usual, we can write $A=A_{1}\dotplus A_{2}$, where $A_{i}\in U_{i}$,
and we may assume that $A\cong A_{1}$. But $U_{1}$ is invariant
under homeomorphisms of $W$, and so $A\in U_{1}$ by~\cite[Lemma~2.2(1)]{Touraille primitive};
hence $I(A)\subseteq U_{1}$ as required. 

It follows easily that if $U\in\mathscr{V}$, then $U=\bigcup\{I(A)\mid A\in\PI(R),A\subseteq U\}$,
and hence that $\mathscr{V}_{\mathrm{cji}}$ is a basis for $\mathscr{V}$.
\end{proof}
\begin{remark}
Touraille's result, translated into topological language, concerned
the case where $\mathscr{V}$ is the smallest complete Heyting{*}
subalgebra of $\mathscr{H}(W)$ and where $W$ is compact. The following
extension of Touraille's final result~\cite[Proposition~4.2]{Touraille primitive}
completes the link between the Heyting{*} algebra and topological
Boolean algebra viewpoints of primitive spaces, and is the dual result
to that of Hansoul (our Theorem~\ref{Fi=00003Dprime TBA}).
\end{remark}

\begin{corollary}[Touraille]
If $W$ is a primitive $\omega$-Stone space, then $\Inv(W)$ is
the smallest complete Heyting{*} subalgebra of $\mathscr{H}(W)$.
\end{corollary}

\begin{proof}
Let $\mathscr{V}$ be the smallest complete Heyting{*} subalgebra
of $\mathscr{H}(W)$, so that $\mathscr{V}\subseteq\Inv(W)$; let
$P=\mathscr{V}_{\mathrm{cji}}$ denote the cji elements of $\mathscr{V}$,
and let $\mathscr{Y}=\{Y_{V}\mid V\in P\}^{*}$ be the trim $P$-partition
of $W$ defined in Proposition~\ref{Heyting general}. If $V\in P$,
then $A\subseteq V$ for all $A\in\Trim_{V}(\mathscr{Y})$ (using~TP6
and~\cite[Lemma~2.2(1)]{Touraille primitive}) and so all $V$-trim
sets are supertrim; hence $V=\bigcup\Trim_{V}(\mathscr{Y})$, by Proposition~\ref{Heyting general}\ref{enu:cji4}.

For $U\in\Inv(W)$, let $J(U)=\bigcup\{\bigcup\Trim_{V}(\mathscr{Y})\mid V\in T_{\mathscr{Y}}(U)\}$;
we claim that $J(U)=U$. For if $A$ is a compact open subset of $U$,
then we can write $A=A_{1}\dotplus\cdots\dotplus A_{n}$, where each
$A_{i}\in\Trim(\mathscr{Y})$, and so $U\subseteq J(U)$. Conversely,
if $V\in T_{\mathscr{Y}}(U)$ then we can find $A\in\Trim_{V}(\mathscr{Y})$
such that $A\subseteq U$; and now if $B\in\Trim_{V}(\mathscr{Y})$
then $B\cong A$ (TP6) and so $B\subseteq U$ (using~\cite[Lemma~2.2(1)]{Touraille primitive});
hence $J(U)=U$ as claimed.

Therefore for $U\in\Inv(W)$, we have $U=\bigcup\{V\mid V\in T_{\mathscr{Y}}(U)\}\in\mathscr{V}$,
as required.
\end{proof}

\paragraph{Declarations:}
\begin{description}
\item [{Ethical~approval}] Not applicable
\item [{Competing~interests}] Not applicable
\item [{Authors\textquoteright ~contributions}] Not applicable
\item [{Availability~of~data~and~materials}] Not applicable. 
\item [{Funding}] No funding was received for conducting this study or
to assist with the preparation of this manuscript.
\end{description}

\end{document}